\documentclass[12pt, reqno]{amsart}
\usepackage{ifthen,amsfonts,
amsmath,amssymb, enumerate,
epsfig,color,graphicx,cite,bbm,graphics}
\usepackage{stmaryrd}
\usepackage{caption}
\usepackage{a4wide}
\usepackage{xfrac}
\usepackage{mathrsfs,bm,amsthm,mathtools,yfonts}
\usepackage{xcolor}
\mathtoolsset{showonlyrefs,showmanualtags}

\makeatletter

\begingroup

\newtheorem{theorem}{Theorem}[section]
\newtheorem{lemma}[theorem]{Lemma}
\newtheorem{propos}[theorem]{Proposition}
\newtheorem{corol}[theorem]{Corollary}
\endgroup

\newtheorem{definition}[theorem]{Definition}

\newtheorem{remark}[theorem]{Remark}

\def\RR{{\mathbb{R}}}

\def\NN{{\mathbb{N}}}
\newcommand{\eps}{{\varepsilon}}
\newcommand{\mean}[1]{\,-\hskip-1.08em\int_{#1}}
\newcommand{\textmean}[1]{- \hskip-.9em \int_{#1}}
\newcommand{\Haus}[1]{{\mathscr H}^{#1}} 
\newcommand{\PHaus}[1]{{\mathscr P}^{#1}} 

\newcommand{\fA}{{\mathcal A}}
\newcommand{\fB}{{\mathcal B}}
\newcommand{\fC}{{\mathcal C}}
\newcommand{\fD}{{\mathcal D}}
\newcommand{\fE}{{\mathcal E}^\flat}
\newcommand{\fF}{{\mathcal F}}
\newcommand{\fT}{{\mathcal T}}

\newcommand{\fueta}{\frac 14}
\newcommand{\fuetainv}{4}

\renewcommand{\div}{{\text {div}}\,}

\makeindex

\begin{document}

\title[The generalized Caffarelli-Kohn-Nirenberg Theorem]
{The generalized Caffarelli-Kohn-Nirenberg Theorem for the hyperdissipative Navier-Stokes system}
\author{Maria Colombo}
\author{Camillo De Lellis}%
\author{Annalisa Massaccesi}

\begin{abstract}
We introduce a notion of suitable weak solution of the
hyperdissipative Navier--Stokes equations and we achieve a corresponding extension of the regularity theory of 
Caffarelli--Kohn--Nirenberg.
\end{abstract}

\maketitle

\tableofcontents

\section{Introduction}

Let $\alpha\geq 0$ and consider the operator $(- \Delta)^\alpha$
whose Fourier symbol is $|\xi|^{2\alpha}$. For a positive integer $\alpha=k$ the operator reduces, up to a sign, to composing $k$ times the classical Laplacian $- \Delta$. 
In this paper $\alpha$ ranges between $1$ and $2$ and we consider the so-called hyperdissipative Navier--Stokes system in $\RR^3$, which
is the following system of (pseudo) partial differential equations
\begin{equation}\label{e:NS_alfa}
\left\{
\begin{array}{l}
\partial_t u + (u\cdot \nabla) u + \nabla p = - (-\Delta)^{\alpha} u\\ \\
\div u =0\, .
\end{array}\right.
\end{equation}
The system is usually complemented with the initial condition 
\begin{equation}\label{e:Cauchy}
u (\cdot, 0) = u_0\, .
\end{equation}
For $\alpha \geq \frac{5}{4}$ and for smooth $u_0$ which decay sufficiently fast at infinity, it is known that the system \eqref{e:NS_alfa}-\eqref{e:Cauchy} has a classical global in time solution, see in particular \cite{Lions}. On the periodic torus a simple proof has been given in \cite{MattinglySinai} when $\alpha > \frac{5}{4}$, whereas the recent papers \cite{Tao} and \cite{BMR} improve the case $\alpha = \frac{5}{4}$ allowing operators with symbols $- |\xi|^{\frac{5}{2}} f (\xi)$ for suitable logarithmic-like $f$. 

\medskip

In this paper we restrict our considerations to the case 
\begin{equation}\label{e:<1/4}
1 < \alpha \leq \frac{5}{4}\,.
\end{equation}
When $\alpha < \frac{5}{4}$ the global existence of classical solutions is still an open question which covers one of the celebrated Millennium Prize problems (see \cite{Fefferman}). In his groundbreaking work \cite{Leray} Leray constructed some global weak solutions (called nowadays Leray--Hopf weak solutions) when $\alpha=1$ and showed several remarkable facts for them. Two are particularly relevant for our discussion:
\begin{itemize}
\item[(i)] the Leray--Hopf weak solutions $(u, p)$ coincide with the classical solutions as long as the latter exist (weak-strong uniqueness);
\item[(ii)] they are regular except for a closed set of  exceptional times, from now on denoted by ${\rm Sing}_T\, u$, which has $0$ Hausdorff $\mathcal{H}^{\sfrac{1}{2}}$ measure.
\end{itemize}
Leray's approach can be carried on to the hyperdissipative Navier--Stokes. Indeed the existence and weak-strong uniqueness are rather straightforward, whereas a suitable generalization of the estimate on the size of ${\rm Sing}_T\, u$ has been given recently in \cite{JiuWang} (see below for
the precise statement). 

\medskip

Following the pioneering work of Scheffer, see \cite{Scheffer1,Scheffer2}, Caffarelli, Kohn and Nirenberg in \cite{CKN} gave a space-time version of the regularity theorem of Leray: they proved, in particular, the existence of global Leray--Hopf solutions which are regular outside a bounded relatively closed set of Hausdorff $\mathcal{H}^1$ measure zero in $\mathbb R^3 \times (0, \infty)$. Indeed their theorem yields a stronger information, see below for the precise statement. Note moreover that, if the initial data is regular enough, Leray's theory implies the regularity of the solution in
a sufficiently small stripe $\mathbb R^3 \times [0, \varepsilon]$ and thus the singular set is also compact. 

In \cite{KP} Katz and Pavlovi\'c gave a first version of the Caffarelli--Kohn--Nirenberg theorem in the range \eqref{e:<1/4}. More precisely they proved that, if
a classical solution blows up at a finite time $T$, then it can be extended smoothly to $(\mathbb R^3\setminus K)\times \{T\}$ for some closed set $K$ of Hausdorff dimension at most $5-4\alpha$. The theorem of Katz and Pavlovi\'c is not a full extension of the Caffarelli--Kohn--Nirenberg: first of all the latter goes beyond the first singular time and secondly the proof of Katz and Pavlovi\'c does not imply $\mathcal{H}^{5-4\alpha} (K)=0$.

\subsection{The extension of the Caffarelli--Kohn--Nirenberg theory}
In the present paper we prove a stronger version of the Katz--Pavlovi\'c result which extends the Caffarelli--Kohn--Nirenberg theorem in its full power. 
In order to give a precise statement we introduce the usual space-time cylinders
\[
Q_r (x_0, t_0) = B_r (x_0) \times (t_0-r^{2\alpha}, t_0]\, ,
\]
compatible with the scaling of the equations (we omit the centers of the ball and the centroids of the cylinder when $x_0=0$ and $(x_0, t_0) = (0,0)$, respectively).
We then define the parabolic Hausdorff measures with the usual Carath\'eodory construction (cf. \cite[2.10.1]{Federer}).
Given $E\subset\RR^3\times\RR$, $\beta\ge 0$ and $\delta>0$, we set
\[
\PHaus{\beta}_{\delta}(E):=\inf\left\{\sum_i r_i^\beta:\,E\subset\bigcup_i Q_{r_i}(x_i,t_i)\text{ and } r_i<\delta\ \forall\,i\right\}\, 
\]
and we call \emph{parabolic Hausdorff measure} of the set $E$ the number
\[
\PHaus{\beta}(E):=\lim_{\delta\to 0}\PHaus{\beta}_\delta(E)=\sup_{\delta>0}\PHaus{\beta}_\delta(E)\,.
\]
Moreover, given a Leray--Hopf weak solution $(u,p)$, we call a point $(x,t)$ {\em regular} if there is a cylinder $Q_r (x,t)$ where $u$ is continuous and we denote by ${\rm Sing}\, u$ the (relatively closed) set of singular points, namely those points which are not regular. 

\begin{theorem}\label{t:main}
Given $\alpha \in (1, \frac{5}{4}]$ and any divergence-free initial data $u_0\in L^2$ there are a Leray--Hopf weak solution $(u,p)$ of \eqref{e:NS_alfa} (see Definition \ref{d:LH}) and a relatively closed set ${\rm Sing}\, u  \subset \mathbb R^3 \times (0, \infty)$ such that $\mathcal{P}^{5-4\alpha} ({\rm Sing}\, u) =0$.
\end{theorem}

For $\alpha=1$ the statement above coincides with the one of Caffarelli, Kohn and Nirenberg and for $\alpha \in (\frac{3}{4}, 1)$ the same result has been shown by Tang and Yu in \cite{TangYu}. 
Given $E\subset\RR^3\times\RR$, it is trivial to prove that
\[
\Haus{\frac \beta 2}\left(\left\{t:\,(\RR^3\times\{t\})\cap E\neq\emptyset\right\}\right)\le C_\beta\PHaus{\beta}(E)
\]
and
\[
\Haus{\beta}\left((\RR^3\times\{t\})\cap E\right)\le C_\beta\PHaus{\beta}(E)\quad\forall\,t\in\RR\,.
\]
We thus recover:
\begin{itemize}
\item a strengthened version of the theorem of Katz and Pavlovi\'c: if a classical solution blows up, the singular set at the first blow-up time has $0$ Hausdorff $\mathcal{H}^{5-4\alpha}$ measure (recall that classical solutions and Leray--Hopf weak solutions coincide as long as the first exist); 
\item the existence of regular solutions for $\alpha = \frac{5}{4}$, because $\mathcal{P}^0$ is the ``counting measure'' and hence for $\alpha=\frac{5}{4}$ the singular set is empty;
\item the generalized Leray's bound on singular times given in \cite{JiuWang}, namely $\mathcal{H}^{(5-4\alpha)/2} ({\rm Sing}_T\, u) =0$ (however the result in \cite{JiuWang} is stronger, since it is proved for {\em any} Leray--Hopf weak solution).
\end{itemize}  

As it is the case for the Scheffer and Caffarelli--Kohn--Nirenberg regularity theory, our theorem is in fact more general. In particular we can introduce a suitable notion of weak solution, which generalizes the one of Caffarelli--Kohn--Nirenberg and which we therefore call as well {\em suitable weak solution}. We then prove their existence and their regularity independently. In particular the bound of Theorem \ref{t:main}, and its consequences, hold for {\em any} suitable weak solution. 

\begin{theorem}\label{thm-intro:sing-set-dim}
Let $(u,p)$ be a suitable weak solution of \eqref{e:NS_alfa} in $\mathbb R^3 \times (0, T)$ as in Definition \ref{def:sws}. Then
\begin{equation}\label{est:dim_singset}\mathcal P ^{5-4\alpha}({\rm Sing \, u})= 0.
\end{equation}
\end{theorem}

Our theorem is in part inspired by the paper of Tang and Yu \cite{TangYu}, where the authors consider the hypodissipative range $\alpha \in (\frac{3}{4}, 1)$. As in their case, our notion of suitable weak solution uses the extension idea introduced by Caffarelli and Silvestre in \cite{CS} to deal with the fractional Laplacian, more specifically we take advantage of a suitable version for exponents $\alpha \in (1,2)$, introduced by Yang in \cite{Yang}. However there are important differences between our work and \cite{TangYu}. 

\subsection{$\varepsilon$--regularity theorem and stability of regular points} One part of our argument for Theorem \ref{thm-intro:sing-set-dim} has an independent interest. As already mentioned, the paper by Caffarelli, Kohn and Nirenberg built on previous works of Scheffer \cite{Scheffer1,Scheffer2}, where the author proved the very first space-time partial regularity result for the Navier--Stokes equations. In particular he proved the $\varepsilon$-regularity statement which is still the starting point of most of the works in the area.
We establish here a suitable generalization of Scheffer's theorem. In what follows, if $f: \mathbb R^3 \times (0, T) \to [0, \infty)$, $\mathcal{M} f$ denotes the maximal function
\[
\mathcal{M} f (x,t) = \sup_{r>0} \frac{1}{r^3} \int_{B_r (x)} f (y,t)\, dy\, .
\]

\begin{theorem}\label{thm:eps_reg-massimale}
There exist positive constants $\varepsilon>0$ and $\kappa>0$ depending only on $\alpha$ such that, if the pair $(u,p)$ is a suitable weak solution of \eqref{e:NS_alfa} in the slab $\RR^3 \times(-r^{2\alpha} + t_0, t_0]$ and satisfies
\begin{equation}\label{eccessivoealternativo}
\frac{1}{r^{6-4\alpha}} \int_{Q_{2r} (x_0, t_0)} \left({\mathcal M}|u|^2+|p|\right)^{\sfrac{3}{2}} \,dx\,dt<\varepsilon\,,
\end{equation}
then $u\in C^{\kappa}\left(Q_r (x_0, t_0)\right)$. For $r\leq 1$ we have an explicit estimate of the form $\|u\|_{C^{\kappa}} \leq C r^{1- 2\alpha (\kappa +1)}$. Moreover the constants are independent of $\alpha \in (1, \frac{5}{4}]$ in the following sense: if
$\alpha \in [\alpha_0, \frac{5}{4}] \subset (1,2]$, then there are positive $\eps_0 (\alpha_0), \kappa_0 (\alpha_0), C (\alpha_0)$ such that $\varepsilon (\alpha) > \varepsilon_0 (\alpha_0)$, $\kappa (\alpha) > \kappa_0 (\alpha_0)$ and $C\leq C (\alpha_0)$. 
\end{theorem}
With Theorem \ref{thm:eps_reg-massimale} it is possible to estimate the box-counting dimension of the singularity, showing in particular that the box-counting dimension of the singular set converges to $0$ when $\alpha\to 5/4$.

\begin{corol}\label{c:box}
If $(u,p)$ is a suitable weak solution of \eqref{e:NS_alfa} in $\mathbb R^3\times (0,T)$, then for every positive $t>0$ the box-counting dimension of
${\rm Sing}\, u\cap (\RR^3 \times [t, \infty))$ is at most ${\frac{15 - 2\alpha - 8\alpha^2}{3}}$.
\end{corol}
Moreover, Theorem \ref{thm:eps_reg-massimale} has the interesting consequence that the set of regular points is stable under perturbations. 

\begin{corol}\label{c:stab}
Assume that:
\begin{itemize}
\item $\alpha_k \to \alpha \in (1, \frac{5}{4}]$, $(u_k, p_k)$ is a suitable weak solution of \eqref{e:NS_alfa} with $\alpha = \alpha_k$ on $\mathbb R^3 \times (0, T)$;
\item $(u,p)$ is a suitable weak solution on $\mathbb R^3 \times (0,T)$ of \eqref{e:NS_alfa};
\item $u_k \rightharpoonup u$ in $L^2 (\RR^3 \times (0, T))$;
\item $u$ is bounded in some $Q_{4r} (x_0, t_0)$.
\end{itemize}
Then for $k$ large enough $u_k$ is H\"older continuous in $Q_r (x_0, t_0)$.
\end{corol}

In turn the latter statement, combined with the Leray weak-strong uniqueness and the existence of smooth solutions at the threshold $\alpha = \frac{5}{4}$, allows, in a suitable sense, to extend the existence of smooth global solutions slightly below $\frac{5}{4}$. One possible formulation is the following. 

\begin{corol}\label{c:reg}
Let $X\subset L^2 (\mathbb R^3;\RR^3)$ be a Banach space of divergence-free vector fields satisfying the following two requirements:
\begin{itemize}
\item[(LT)] For any $u_0\in X$ and any $\alpha \in [\frac{9}{8}, \frac{5}{4}]$ there is a classical solution of \eqref{e:NS_alfa}-\eqref{e:Cauchy} on a time
interval $[0, T( \|u_0\|_X)]$ depending only upon $\|u_0\|_X$ and not upon $\alpha$. 
\item[(WS)] Any Leray--Hopf weak solution has to coincide with the solution of {\rm (LT)} on the interval of existence of the latter.
\end{itemize}
Then, for any bounded set $Y\subset X$ there is $\alpha_0 (Y)< \frac{5}{4}$ with the following property: for any $\alpha \in [\alpha_0, \frac{5}{4}]$ and $u_0\in Y$ there is a unique Leray--Hopf weak solution of \eqref{e:NS_alfa}-\eqref{e:Cauchy} on $\mathbb R^3 \times (0, \infty)$, which in addition is smooth. 
\end{corol}

Any space $X$ of functions which are regular enough satisfies (LT) and (WS). Indeed the amount of regularity needed is not much: for instance, with a slight modification of Leray's original arguments, it is
not difficult to see that $H^1$ fulfills both conditions. We thus conclude that for any $H^1$ initial data there are
global smooth solutions of \eqref{e:NS_alfa} whenever $\alpha$ is sufficiently close to $\frac{5}{4}$ (and this closeness is uniform on bounded subsets of $H^1$). {An analogous result regarding the existence of smooth solutions for the slightly supercritical surface quasi-geostrophic equation on bounded subsets of a suitably chosen Banach space was obtained in \cite{cotivicol} with different methods.}

\subsection{Plan of the paper} In Section \ref{s:suitable} we will discuss the various notions of weak solutions used in the note, their existence and the two main $\varepsilon$-regularity statements, namely Theorem \ref{thm:eps_reg-variant} and Theorem \ref{t:CKN}, from which all the results claimed so far will be derived. The Sections \ref{s:energy}, \ref{s:compactness} and \ref{s:linear} outline the three main ingredients of the proof of the first $\varepsilon$-regularity Theorem \ref{thm:eps_reg-variant}. In particular:
\begin{itemize}
\item Section \ref{s:energy} derives the most important consequence of our definition of suitable weak solution, namely a local energy estimate (cf. Lemma \ref{lemma:suitable}), which in turn implies a crucial compactness property of solutions, proved in Section \ref{s:compactness}, cf. Lemma \ref{lem:lions}. 
\item Section \ref{s:linear} discusses the interior H\"older regularity of the solutions of a suitable linearization of \eqref{e:NS_alfa}, cf. Lemma \ref{lem:hoelder}.
\end{itemize}
The compactness lemma and the estimates on the linearized system are then combined in Section \ref{s:excess_decay} to
prove a suitable excess decay property of solutions of \eqref{e:NS_alfa}, cf. Proposition \ref{prop:exc_decay}. The latter proposition is iterated to prove Theorem \ref{thm:eps_reg-variant}. In Section \ref{s:CKN} we show then how to derive the second $\varepsilon$-regularity Theorem \ref{t:CKN} from Theorem \ref{thm:eps_reg-variant}. Finally in Section \ref{s:final} we prove the various results claimed above. 

\subsection{Acknowledgments and final remarks} The first author acknowledges the support of Dr. Max R\"ossler, of the Walter Haefner Foundation and of the ETH Z\"urich Foundation. Part of this work has been carried on while the second author was spending a semester at the CMSA at Harvard and a one month visit at the IHES at Bures sur Yvette. This project has received funding from the European Union's Horizon 2020 research and innovation programme under the grant agreement No.752018 (CuMiN). 

The authors wish to thank Joachim Krieger, Wilhelm Schlag and Andrea Nahmod for very useful advices and conversations at a very early stage of their work and Alice Chang and Vlad Vicol for useful comments towards the end. 

While we were completing this manuscript we learned of a similar independent work in preparation
of Eric Chen, cf. \cite{chen-forth}, aimed at establishing the main conclusion of our paper, namely
the estimate on the singular set of Theorem \ref{thm-intro:sing-set-dim}.

\section{Suitable weak solutions}\label{s:suitable}

\subsection{Leray--Hopf weak solutions} We introduce the usual concept of Leray--Hopf weak solutions. 

\begin{definition}\label{d:LH} Let $u_0\in L^2 (\RR^3)$ be a divergence-free vector field. A pair $(u,p)$ is a Leray--Hopf weak solution of \eqref{e:NS_alfa}-\eqref{e:Cauchy}  on $\mathbb R^3\times (0,T)$ if:
\begin{itemize}
\item[(a)] $u\in L^\infty ((0,T), L^2 (\RR^3)) \cap L^2 ((0,T), H^\alpha (\RR^3))$;
\item[(b)] $u$ solves \eqref{e:NS_alfa}-\eqref{e:Cauchy} in the sense of distributions, namely ${\rm div}\, u = 0$ and 
\[
\int \Big(\partial_t \varphi \cdot u + \sum_{i,j} u^iu^j \partial_j \varphi^i - (-\Delta)^\alpha \varphi \cdot u\Big)\, dx\, dt 
= - \int u_0 (x) \cdot \varphi (0,x)\, dx
\]
for every divergence-free $\varphi\in C^\infty_c (\mathbb R^3\times \mathbb R, \mathbb R^3)$.
\item[(c)] $p$ is the potential-theoretic solution of $-\Delta p = {\rm div}\, {\rm div}\, (u\otimes u)$ (here and in the rest of the paper we use the notation ${\rm div}\, {\rm div}\, (u\otimes u)$ for $\sum_{i,j} \partial^2_{ij} (u_i u_j)$);
\item[(d)] The following inequalities hold:
\begin{align}
&\frac{1}{2} \int |u|^2 (x,t)\, dx + \int_0^t \int |(-\Delta)^{\sfrac{\alpha}{2}} u|^2 (x,\tau)\, dx\, d\tau \leq \frac{1}{2} \int |u_0|^2 (x)\, dx
\qquad \forall t>0\label{e:g_energy_1}\\
&\frac{1}{2} \int |u|^2 (x,t)\, dx + \int_s^t \int |(-\Delta)^{\sfrac{\alpha}{2}} u|^2 (x,\tau)\, dx\, d\tau \leq \frac{1}{2} \int |u|^2 (x,s)\, dx
\quad \mbox{for a.e. $s$ and every $t>s$.} \label{e:g_energy_2} 
\end{align}
\end{itemize}
\end{definition}

As already mentioned, we have the following simple extension of Leray's theory:

\begin{theorem}\label{t:Leray}
For any divergence-free $u_0\in L^2 (\RR^3)$ there is a Leray--Hopf weak solution of \eqref{e:NS_alfa}-\eqref{e:Cauchy} on $\RR^3 \times (0, \infty)$. Moreover, if $v$ is a second
Leray--Hopf weak solution such that $\|v (\cdot, t)\|_\infty \in L^2 (0,T)$ then $u=v$ on $\mathbb R^3\times (0,T)$. 
\end{theorem}

The proof follows the idea of Leray's paper \cite{Leray} with minor modifications: we refer the reader to the Appendix for a detailed proof. 

\subsection{Caffarelli-Silvestre and Yang extension problems} In order to define suitable weak solutions, we deal with the fractional Laplacian by defining suitable extensions to the half space $\mathbb R^{n+1}_+ := \mathbb R^n \times [0, \infty)$ of functions defined on $\mathbb R^n$. In what follows we always use the variable $x$ for the factor $\mathbb R^n$ and the variable $y$ for the factor $[0, \infty)$. Moreover we use the notation $\overline \nabla$, $\overline \Delta$ and so on for differential operators defined on $\mathbb R^{n+1}_+$.
The following theorem is essentially due to Yang \cite{Yang} (see also the survey \cite{chang-yang}): for our analysis we need however some adjustments of his statements and some further properties and for this reason we give a self-contained proof of the theorem and of other related facts in Appendix~\ref{a:ext}.

\begin{theorem}[Yang, 2013]\label{thm:yang}
	Let $u\in H^\alpha(\RR^n)$, with $\alpha\in (1,2)$ and set $b:=3-2\alpha$. Define the differential operator $\overline \Delta_b$ as
	\begin{equation}\label{calc:lapl_b}
	\overline \Delta_b u^*:=\overline \Delta u^*+\frac{b}{y}\partial_y u^* = \frac{1}{y^b} \overline{{\rm div}}\, \big(y^b \overline\nabla u^*\big)\,.
	\end{equation}
Then there is a unique ``extension'' $u^*$ of $u$ in the weighted space $L^2_{\rm loc}(\RR^{n+1}_+,y^b)$ which satisfies $\overline\Delta_b u^*\in L^2(\RR^{n+1}_+,y^b)$ and 
	\begin{equation}\label{biharm}
	\overline{\Delta}^2_b u^*(x,y)=0
	\end{equation}
	and the boundary conditions
	\begin{align}
	u^*(x,0) &=u(x)\label{e:Dirichlet}\\
	\lim_{y\to 0}y^{1-\alpha} \partial_y u^*(x,y) &=0\, .\label{e:Neumann}
	\end{align}
Moreover, there exists a constant $c_{n, \alpha}$, depending only on $n$ and $\alpha$, with the following properties:
\begin{itemize}
\item[(a)] The fractional Laplacian $(-\Delta)^\alpha u$ is given by the formula 
\begin{equation}
	\label{eqn:frac-lap-est}
		(-\Delta)^\alpha u(x)=c_{n,\alpha}\lim_{y\to 0}y^b\partial_y\overline{\Delta}_bu^*(x,y)\, .
	\end{equation}
\item[(b)] The following energy identity holds
\begin{equation}\label{e:en_of_ext}
\int_{\RR^n}|(-\Delta)^{\frac{\alpha}{2}} u|^2\,dx=\int_{\RR^n}|\xi|^{2\alpha}|\widehat{u}(\xi)|^2\,d\xi=c_{n,\alpha}\int_{\RR^{n+1}_+}y^b|\overline\Delta_b u^*|^2\,dx\,dy\, .
\end{equation}
\item[(c)] The following inequality holds for every extension $v\in L^2_{\rm loc} (\RR^{n+1}_+, y^b)$ of $u$ with $\overline\Delta_b v\in L^2(\RR^{n+1}_+,y^b)$
:
\begin{equation}\label{e:minimum_yang_ext}
\int_{\RR^{n+1}_+}y^b|\overline\Delta_b u^*|^2\,dx\,dy \leq \int_{\RR^{n+1}_+}y^b|\overline\Delta_b v|^2\,dx\,dy .
\end{equation}
\end{itemize}
\end{theorem}

\begin{remark}\label{r:lisce}
The boundary conditions \eqref{e:Dirichlet} and \eqref{e:Neumann} can be really taken pointwise if $u$ is (sufficiently) smooth, otherwise they must be understood in the sense of distributions. In particular, we will prove a suitable representation formula for $u^*(x,y)$ as $(P (\cdot, y) \ast u) (x)$, where $P$ is the Poisson-type kernel of \eqref{e:Poisson}: corresponding boundary conditions will be derived for $P$ in the distributional sense. When $u$ is less regular, even though Theorem \ref{thm:yang} does not apply, we will still use the notation $u^*$ for its extension $P (\cdot, y)\ast u$.
\end{remark}

\begin{remark}\label{r:a_b}
Observe that, since $2-\alpha>0$, \eqref{e:Neumann} implies also the boundary condition appearing in \cite{Yang}
\begin{equation}\label{e:Neumann_2}
\lim_{y\to 0} y^b {\partial_yu^*} (x,y) = 0\, .
\end{equation}
\end{remark}

Yang's theorem is an extension to higher order operators of a theorem of Caffarelli and Silvestre, cf. \cite{CS}. The latter will also turn out useful in our considerations and we therefore recall it here.

\begin{theorem}[Caffarelli--Silvestre]\label{thm:CF}
	Let $w\in H^{\alpha-1} (\RR^n)$, with $\alpha\in (1,2)$ and set $b:=3-2\alpha$. Then there is a unique ``extension'' $w^\flat$ of $w$ in the weighted space $H^1(\RR^{n+1}_+,y^b)$ which satisfies 
	\begin{equation}\label{harm}
	\overline{\Delta}_b w^\flat(x,y)=0
	\end{equation}
	and the boundary condition
	\begin{equation}\label{eqn:b-c-caffarelli}
	w^\flat(x,0)=w(x)\, .
	\end{equation}
Moreover, there exists a constant $C_{n, \alpha}$, depending only on $n$ and $\alpha$, with the following properties:
\begin{itemize}
\item[(a)] The fractional Laplacian $(-\Delta)^{\alpha-1} w$ is given by the formula 
\begin{equation}
	\label{eqn:frac-lap-est-CS}
		(-\Delta)^{\alpha-1} w (x)=C_{n,\alpha}\lim_{y\to 0}y^b\partial_y w^\flat (x,y)\,.
	\end{equation}
\item[(b)] The following energy identity holds
\begin{equation}\label{e:en_of_ext-CS}
\int_{\RR^n}|(-\Delta)^{\frac{\alpha-1}{2}} w|^2\,dx=\int_{\RR^n}|\xi|^{2\alpha-2}|\widehat{w}(\xi)|^2\,d\xi=C_{n,\alpha}\int_{\RR^{n+1}_+}y^b|\overline\nabla w^\flat|^2\,dx\,dy\,.
\end{equation}
\item[(c)] The following inequality holds for every extension $v\in H^1 (\RR^{n+1}_+, y^b)$ of $w$:
\begin{equation}\label{e:minimum_ext_caff}
\int_{\RR^{n+1}_+}y^b|\overline \nabla w^\flat|^2\,dx\,dy \leq \int_{\RR^{n+1}_+}y^b|\overline \nabla v|^2\,dx\,dy\,.
\end{equation}\
\end{itemize}
\end{theorem}
In the rest of the note we will always use $b$ to denote the exponent $3-2\alpha$ in the operator $\overline{\Delta}_b$, in the measure $y^b\, dx\, dy$ and in the corresponding weighted Sobolev spaces. The same exponent appears, however, in some other instances, more precisely in the scaling of certain integral quantities and in the summability of $u$ and $p$: in such cases we will use, instead, $3-2\alpha$.

\subsection{Suitable weak solutions} We are now ready to define suitable weak solutions. In the next formula and in the rest of the paper we use Einstein's convention for the sum on repeated indices.
\begin{definition}\label{def:sws}{\rm
A Leray--Hopf weak solution $(u,p)$ on $\RR^3 \times [0,T]$ is a suitable weak solution if the following inequality 
holds for a.e. $t\in [0,T]$ and all nonnegative test functions\footnote{That is, the function $\varphi$ vanishes when 
$|x| + y + |t|$ is large enough and if $t$ is sufficiently close to $0$, but it can be nonzero on some regions of $\{(x,y,t): y=0\}$.} $\varphi \in C^{\infty}_c ( \RR^4_+ \times (0,T))$ with $\partial_y \varphi (\cdot,0,\cdot)= 0$ in $\RR^3 \times (0,T)$:
	\begin{align}
\int_{\RR^3} \varphi (x,0,t) \frac{|u(x,t)|^2}{2} \, dx +c_{\alpha}&\int_0^t \int_{\RR^4_+} y^b|\overline \Delta_b u^* |^2 \varphi \, dx \, dy \, ds\nonumber\\
	\leq&\;
	\int_0^t \int_{\RR^3} \left[ \frac{|u|^2}{2} \partial_t \varphi|_{y=0} + \Big(\frac{| u |^2}{2} +p\Big) u \cdot \nabla \varphi|_{y=0} \right] \, dx \, ds\nonumber\\
 &\;-c_{\alpha} \int_0^t\int_{\RR^4_+}y^b\overline\Delta_b u^*_i\left(2\overline\nabla\varphi\cdot\overline\nabla u^*_i+u^*_i \overline\Delta_b\varphi \right)\, dx \, dy \, ds \,,\label{eqn:suit-weak-tested}
	\end{align}
where the constant $c_{\alpha}$ depends only on $\alpha$ and comes from Theorem~\ref{thm:yang}.
}\end{definition}

It is not difficult to see that for a {\em smooth} solution of \eqref{e:NS_alfa} the inequality \eqref{eqn:suit-weak-tested} holds indeed with the equality sign (for the reader's convenience we have included the proof in the appendix). A minor modification of the proof that Leray--Hopf weak solutions exist yields the existence of suitable weak solutions as well. Again, we refer to the appendix for the complete proof:

\begin{theorem}\label{t:suitable}
For any divergence-free $u_0\in L^2 (\RR^3)$ there is a suitable weak solution of \eqref{e:NS_alfa}-\eqref{e:Cauchy} on $\RR^3 \times (0, \infty)$.
\end{theorem}

\subsection{Statements of the $\eps$-regularity Theorems}
We are finally ready to state our main $\eps$-regularity statements.  First of all Theorem \ref{thm:eps_reg-massimale} will be derived from a similar one where the smallness assumption is in fact weaker. In order to state it we need to introduce a suitable ``tail functional'':
\begin{equation}\label{def:tail}
\fT(u;x,t,r) := r^{5\alpha -2} \int_{t-r^{2\alpha}}^t \sup_{R\geq \frac{r}{4}} \frac{1}{R^{3\alpha}} \mean{B_R(x)} |u(y,t)|^{2}\, dy\, dt\,.
\end{equation}

\begin{theorem}\label{thm:eps_reg-variant}
	There exist positive constants $\varepsilon$ and $\kappa$, depending only on $\alpha$, such that, if  $(u,p)$ is a suitable weak solution of the hyperdissipative Navier-Stokes equations \eqref{e:NS_alfa} in the slab $\RR^3 \times (-2^{2\alpha},0)$ and satisfies
	\begin{equation}\label{eccessivoealternativo-variant}
	\int_{Q_2} \left( |u|^{3}+|p|^\frac 32\right) \,dx\,dt + \fT(u;0,0,2)<\varepsilon\,,
	\end{equation}
	then $u\in C^{0,\kappa}\left(Q_1;\RR^3\right)$. Moreover, the constants are independent of $\alpha$ in the sense of Theorem \ref{thm:eps_reg-massimale}. 
\end{theorem}

The second statement uses the Caffarelli-Silvestre extension $(\nabla u)^\flat$ of $\nabla u$. 
Observe that
\begin{equation*}
c_\alpha\int_{\mathbb R^4_+} y^{b} |\overline\nabla (\nabla u)^\flat|^2 (x,y,t)\, dx\, dy 
= \int_{\mathbb R^3} \left|(-\Delta)^{\frac{\alpha-1}{2}} \nabla u\right|^2 (x,t)\, dx=  \int_{\mathbb R^3} |(-\Delta)^\frac{\alpha}{2} u|^2 (x,t)\, dx\, .
\end{equation*}
In particular we easily conclude that, for a Leray--Hopf weak solution on $\RR^3\times (0,T)$, 
\begin{equation}\label{e:energia_finita}
c_\alpha\int_0^T\int_{\mathbb R^4_+} y^{b} |\overline\nabla (\nabla u)^\flat|^2 (x,y,t)\, dx\, dy\, dt
\leq \frac{1}{2} \int_{\RR^3} |u_0|^2 (x)\, dx < \infty\, .
\end{equation}
We introduce a suitable localized and rescaled version of the left hand side. In particular, we define first a suitable counterpart of the parabolic cylinders in
the domain $\RR^4_+\times \RR$:
\[
Q^*_r (x,t) = B_r (x) \times [0,r) \times (t-r^{2\alpha}, t]\, 
\]
and we then consider the quantity
\begin{align*}
\fE(u;x,t, r) & :=\frac{1}{r^{5-4\alpha}}\int_{Q^*_r(x,t)} y^{b}|\overline\nabla ( \nabla u)^\flat |^2\,dx\,d\tau.
\end{align*}
\begin{theorem}\label{t:CKN}
	There exists $\delta>0$ such that, if $(u,p)$ is a suitable weak solution of \eqref{e:NS_alfa} in $\RR^3\times I$ and, at some point $(x,t)$, we have
	\begin{equation}\label{eps_reg:hp2}
	\limsup_{r\to 0}\fE(u;x,t, r) <\delta
	\end{equation}
	then $(x,t)$ is a regular point. 
\end{theorem}

\section{The energy inequality}\label{s:energy}

In this section we consider the hyperdissipative Navier-Stokes equation with a constant drift term $M\in\RR^3$ and a  scalar factor $L\in\RR$ multiplying the quadratic nonlinearity, namely 
\begin{align}\label{e:sws-const}
\left\{
\begin{array}{l}
\partial_t u + ((Lu+M)\cdot \nabla) u + \nabla p = - (-\Delta)^{\alpha} u\\ \\
\div u =0\,. 
\end{array}\right.
\end{align}

\noindent The notion of suitable weak solution generalizes trivially to this context in the following way: for a.e. $t\in [0,T]$ and all nonnegative test functions $\varphi \in C^{\infty}_c ( \RR_4^+ \times (0,T))$ with 
\begin{equation}\label{e:bound_test}
\partial_y \varphi (\cdot,0,\cdot)= 0 \qquad \mbox{on $\RR^3 \times (0,T)$}
\end{equation} 
we require that a suitable weak solution of \eqref{e:sws-const} satisfies
	\begin{multline}\label{eqn:suit-weak-drift}
	\int_{\RR^3} \varphi(\cdot,0,t) \frac{|u(\cdot,t)|^2}{2}
	 +c_{\alpha}\int_0^t \int_{\RR^4_+} y^b|\overline \Delta_b u^* |^2 \varphi 
	\leq 
	\int_0^t \int_{\RR^3} \left[ \frac{|u|^2}{2} \partial_t \varphi(\cdot,0,\cdot) + \Big(\frac{L| u |^2}{2} +p\Big) u \cdot \nabla \varphi(\cdot,0,\cdot) \right]
	\\
 +\int_0^t\int_{\RR^3}\frac{|u|^2}{2} M\cdot\nabla\varphi(\cdot,0,\cdot)
 -c_{\alpha}\int_0^t\int_{\RR^4_+}y^b\left(2\overline\nabla\varphi\overline\nabla u^* \overline\Delta_b u^* +u^* \overline\Delta_b\varphi\overline\Delta_b u^*\right)
	\end{multline}
(the integrals in space have to be intended in $dx$ when the domain of integration is $\RR^3$, and in $dx\,dy$ when the domain of integration is $\RR^4_+$; the constant $c_{\alpha}$ depends only on $\alpha$ and comes from Theorem~\ref{thm:yang}).
The system arises naturally because we will subtract constants from suitable weak solutions of the hyperdissipative Navier-Stokes equations and we will rescale them.
In particular we have the following

\begin{lemma}\label{lemma:drift}
Let $(u,p)$ be a suitable weak solution of the hyperdissipative Navier-Stokes system. Let $M\in\RR^3$ and $L>0$, let $\tilde p\in L^1((0,T))$. Then $v := (u-M)/L$ and $q:= (p - \tilde p)/L$ solve \eqref{e:sws-const} and satisfy the modified energy inequality \eqref{eqn:suit-weak-drift}.
\end{lemma}
\begin{proof} We first observe that subtracting a function $\tilde{p}$ of time to the pressure does not change the Navier-Stokes equations (where the pressure enters only through its spatial gradient) nor it affects the inequality for suitable weak solution, since it adds a term of the form
\[
- \int_0^t \tilde{p} (s) \int_{\RR^3} u (x, s) \cdot \nabla \varphi (x,0,s)\, dx\, ds
\]
which vanishes because $u$ is divergence-free. Hence, we can assume that $\tilde p \equiv 0$ and that $q= p/L$.
	
Given a nonnegative test function $\varphi \in C^{\infty}_c ( \RR^4_+ \times (0,T))$ with $\partial_y \varphi (\cdot,0,\cdot)= 0$ in $\RR^3 \times (0,T)$, we write the following simple algebraic identity.
\begin{align}
\int_{\RR^3} \varphi(x,0,t) \frac{|v (x,t)|^2}{2} 
 +c_{\alpha}\int_0^t \int_{\RR^4_+} y^b|\overline\Delta_b v^* |^2 &\varphi
  \nonumber
= \; \frac{1}{L^2}\int_{\RR^3} \varphi(x,0,t) \frac{|u(x,t)|^2}{2} 
+\frac{c_{\alpha}}{L^2}\int_0^t \int_{\RR^4_+} y^b|\overline\Delta_b u^*|^2 \varphi
\nonumber
\\
&\; -\frac{1}{L^2} \int_{\RR^3} \varphi(x,0,t) M \cdot u(x,t)
+\frac{1}{L^2}\int_{\RR^3} \varphi(x,0,t)\frac{|M|^2}{2}
 ,\label{e:start_en-ineq-drift}
\end{align}
where we have simply observed that $v^* = \frac{1}{L} (u^* -M)$ and thus 
\begin{equation}\label{e:v-to-u}
\overline\nabla v^* = L^{-1} \overline \nabla u^* \qquad \mbox{and}\qquad \overline\Delta_b v^* = L^{-1} \overline{\Delta}_b u^*\, . 
\end{equation}
We apply the energy inequality to the function $u$ to control the two terms in the second line. We then test the hyperdissipative Navier-Stokes system with $\varphi(\cdot,0,\cdot)M$ to control the first term in the third line:
$$
 \int_{\RR^3} \varphi(\cdot,t) u(\cdot,t)\cdot M \, dx = \int_0^t \int_{\RR^3} \Big[ \partial_t \varphi M\cdot u+  LM\cdot u \nabla \varphi \cdot u+M \cdot \nabla\varphi p- \varphi M\cdot \Delta^\alpha u\big] \, dx \, ds\,.
$$ 
Notice that the last term of the right hand side can be rewritten in terms of the extension $u^*$ by using \eqref{eqn:frac-lap-est}, the divergence theorem and \eqref{biharm}; namely for every $t\in [0,T]$
\begin{align*}
- \int_{\RR^3}\varphi(\cdot,0,t) M\cdot \Delta^\alpha u & =-c_{\alpha}\lim_{y\to 0}\int_{\RR^3}y^b\varphi(\cdot,0,t) M\cdot \partial_y\overline\Delta_bu^*=c_{\alpha}\int_{\RR^4_+} \overline{\rm div}\left(y^b\varphi M_i \overline\nabla\overline\Delta_bu^*_i\right)\\
& =c_{\alpha}\int_{\RR^4_+}y^bM_i\overline\nabla\varphi\cdot\overline\nabla\overline\Delta_b u^*_i =-c_{\alpha}\int_{\RR^4_+}y^b\overline\Delta_b\varphi M\cdot\overline\Delta_b u^*\, ,
\end{align*}
where in the last line we have integrated again by parts taking advantage of \eqref{e:bound_test}. 
To rigorously justify the formulas, we approximate $u$ with smooth functions $u\ast \rho_\theta$, where $\{\rho_\theta\}_{\theta \in (0,1)}\subseteq C^\infty_c(\RR^3)$ is a standard family of mollifiers. Notice that $(u\ast \rho_\theta)^*=u^*\ast \rho_\theta$. For smooth functions the integration by parts can be justified using Remark \ref{r:lisce}.
To take the limit in the identity we observe that all the terms in the right hand side pass to the limit since $u^*\ast \rho_\theta \to u^*$, $\overline \Delta_b u^*\ast \rho_\theta \to \overline \Delta_b u^*$ in $L^2(\RR^4_+, y^b)$ while for the left hand side we employ the distributional convergence of the $\alpha$-Laplacian.

Finally, we rewrite the last term of \eqref{e:start_en-ineq-drift} as 
\[
\int_{\RR^3} \varphi(x,0,t)\frac{|M|^2}{2} \, dx  = \int_{\RR^3}\int_0^t \partial_t\varphi(x,0,s) \frac{|M|^2}{2} \, dx\,ds .
\]
Putting together these estimates and since $u$ is divergence-free, we obtain that
	\begin{multline*}
\int_{\RR^3} \varphi(x,0,t) \frac{|v(x,t)|^2}{2} \, dx +c_{\alpha}\int_0^t \int_{\RR^4_+} y^b|\overline \Delta_b v^* |^2 \varphi \, dx \, dy \, ds \\
\leq 
\int_0^t \int_{\RR^3} \left[ \frac{|v|^2}{2} \partial_t \varphi(x,0,s) +\Big(\frac{|v|^2}{2} u +q v \Big) \cdot \nabla \varphi(x,0,s) \right] \, dx \, ds \\
-\frac{c_{\alpha}}{L^2}\int_0^t\int_{\RR^4_+}y^b\left(2\overline\nabla\varphi\overline\nabla u^* \overline\Delta_b u^* +(u^*-M) \overline\Delta_b\varphi\overline\Delta_b u^*\right)\, dx \, dy \, ds \,,
\end{multline*}
We next add and subtract the term
\[
\int_0^t\int_{\RR^3}\frac{|v|^2}{2} M\cdot\nabla\varphi(\cdot,0,\cdot)\,dx\,ds
\]
and use \eqref{e:v-to-u} (namely $\overline \nabla u^* = L \overline \nabla v^*$ and $\overline{\Delta}_b u^* = L \overline{\Delta}_b v^*$) to conclude \eqref{eqn:suit-weak-drift}.
\end{proof}

In the next key lemma we show that the local energy inequality allows to control a suitably localized energy in terms of lower order norms of $u$. 

\begin{lemma}\label{lemma:suitable}
	Let $M \in \RR^n$, $f\in L^1([0,1])$ and $(u,p)$  be a suitable weak solution of \eqref{e:sws-const} in $\RR^3\times [-1,0]$. 
Then we have that 
	\begin{multline}\label{eqn:loc-en} 
	\sup_{t\in [-(3/4)^{2\alpha},0]}\int_{B_{3/4}} \frac{|u(x,t)|^2}{2}\, dx + \int_{Q^*_{3/4}} y^b | \overline\Delta_b u^*|^2 \, dx\, dy \, dt\\
	\leq C  (1+|M|)
	\int_{Q_{1}} |u|^2 \, dx \, dt+C  
	\int_{Q_{1}} |u|\big[|L| u |^2 -f(t)|+|p|\big]  \, dx \, dt + C \int_{Q_{1}^*} y^b  |u^*|^2 \, dx \, dy\, dt .
	\end{multline}
\end{lemma}

We remark that the lemma will be applied twice: in the proof of Theorem~\ref{thm:eps_reg-variant}, to get strong compactness of the rescaled sequence, with $M=(u)_{Q_1}$, and $f \equiv 0$; and in Theorem~\ref{t:CKN},  with $M=0$, $L=1$, and $f=[u]^2_{B_1}$, where here and in the rest of the note we use the shorthand notation 
\begin{align}
[u]_\Omega (t) & = \mean{\Omega} u (x,t)\, dx\label{e:space_av}\\
(u)_\Gamma & = \mean{\Gamma} u(x,t)\, dx dt\, .\label{e:time_av}
\end{align} 
respectively for space and space-time averages.

The subtraction of the function $f(t)$ in the cubic term of the energy is due to the fact that in the Navier-Stokes equations the nonlinear term has a divergence structure. In order to prove the lemma we need a suitable interpolation inequality, which in fact will prove crucial in several other occasions.
We state it assuming that $u$ and $u^*$ are real-valued functions: in fact, the lemma will be applied componentwise to the velocity field. Since it will be used several times, consistenly with the notation $Q^*_r (x,t)$ we introduce also
\[
B^*_r (x) := B_r (x) \times [0,r[ \subset \RR^4_+
\]
 
\begin{lemma}\label{lem:interpol}
Let $\psi\in C^\infty_c(\RR^4)$, $\varepsilon\in(0,1)$, $r\in (0, \infty)$, $u \in H^{\alpha}(\RR^3)$ and $u^*$ be its extension given by Theorem~\ref{thm:yang}. Then the following inequalities hold for a constant $C$ depending only on $\alpha$:
\begin{equation}\label{e:interp_1}
\int_{\RR^4_+} y^b |\overline\nabla u^*|^2 \psi^2 \, dx\, dy\le \varepsilon \int_{\RR^4_+}y^b|\overline \Delta_b u^*|^2\psi^2\,dx\,dy+\frac{C}{\varepsilon}\int_{\RR^4_+}y^b|u^*|^2\left(\psi^2+|\overline\nabla\psi|^2\right)\,dx\,dy\, ,
\end{equation}
\begin{align}
\int_{B_{r}^*} y^b |\overline\nabla u^*|^2 \, dx\, dy &
\leq C\left( \int_{B_{2r}^*}  y^b |\overline\Delta_b u^*|^2 \, dx\, dy\right)^\frac 12 \left(\int_{B_{2r}^*}  y^b | u^*|^2 \, dx\, dy\right)^\frac 12
+\frac{C}{r^2} \int_{B_{2r}^*}  y^b |u^*|^2 \, dx\, dy.\label{e:interp_2}
\end{align}
\end{lemma}

\begin{proof}
We first assume in addition that $u\in C^\infty$. Remember that $\overline\Delta_b u^*=y^{-b}\overline{\rm div}(y^b\overline\nabla u^*)$, thus
\begin{equation}
\label{eqn:calcolo}
\begin{split}
 \overline\div(y^b u^*\psi^2 \overline\nabla u^* ) &=  
 \overline\div(y^b  \overline\nabla u^*) u^*\psi^2 +
 y^b |\overline\nabla u^*|^2 \psi^2 
 + 2y^b \psi u^* \overline\nabla \psi \cdot \overline\nabla u^*
 \\
 &=  
 y^b \overline\Delta_b u^* u^*\psi^2 +
 y^b |\overline\nabla u^*|^2 \psi^2 
 + 2y^b \psi u^* \overline\nabla \psi \cdot \overline\nabla u^* .
\end{split}
\end{equation}
Integrating by parts the boundary term vanishes thanks to $\lim_{y\to 0}y^b\partial_y u^*(x,y)=0$, thus
\begin{equation}\label{e:interp_ibp}
\int_{\RR^4_+} y^b|\overline\nabla u^*|^2\psi^2 =-\int_{\RR^4_+} y^b u^* \overline\Delta_b u^*\psi^2-2\int_{\RR^4_+} y^b\psi u^*\overline \nabla u^*\cdot\overline\nabla\psi.
\end{equation}
More precisely, to obtain the previous estimate we integrate in $\{y \geq \eps\}$ and then we let $\eps \to 0$; the boundary term satisfies
$$\liminf_{y \to0 } \left|\int_{\RR^3} y^b u^*\psi^2 \partial_y u^*\, dx\right|^2 \leq 
\liminf_{y \to0 } \left(\int_{\RR^3} y^b |u^*|^2\psi^2 \, dx\right) \cdot \left(\int_{\RR^3} y^b \psi^2 |\partial_y u^*|^2\, dx\right)= 0
$$
where the last equality follows from the $L^2$ estimate on $u^*$ of Lemma~\ref{lem:withtails}, which shows that the first factor is bounded for a sequence of $y$ going to $0$, and to the fact that $\lim_{y\to 0}y^b\partial_y u^*(x,y)=0$ uniformly, thanks to the smoothness of $u$ and to \eqref{e:Neumann_2}.
By H\"older inequality, we estimate the first summand in the right hand side of \eqref{e:interp_ibp}
\[
-\int_{\RR^4_+} y^b u^* \overline\Delta_b u^*\psi^2 \le \left(\int_{\RR^4_+}y^b|\overline\Delta_bu^*|^2\psi^2\right)^{\frac 12}\left(\int_{\RR^4_+} y^b|u^*|^2\psi^2\right)^{\frac 12}.
\]
We use Young's inequality to the second summand of \eqref{e:interp_ibp}, thus
\[
\int_{\RR^4_+} y^b\psi u^*\overline \nabla u^*\overline\nabla\psi\le C\int_{\RR^4_+} y^b|u^*|^2|\overline\nabla\psi|^2+\frac 14\int_{\RR^4_+} y^b\psi^2|\overline\nabla u^*|^2.
\]
The second summand of the latter inequality an be reabsorbed in the left hand side of \eqref{e:interp_ibp}. Putting together the last three displayed inequalities and using Young's inequality a second time, we obtain \eqref{e:interp_1}. Applying the same arguments with  a cutoff function $\psi$ between $B_r^*$ and $B_{2r}^*$, we obtain \eqref{e:interp_2}.

This concludes the proof in the case that $u$ is smooth; if $u \in H^\alpha(\RR^3)$ is not assumed to be smooth, we proceed by approximating $u$ with $u\ast \rho_\theta$, where $\{\rho_\theta\}_{\theta \in (0,1)}\subseteq C^\infty_c(\RR^3)$ is a standard family of mollifiers, and we notice that $(u\ast \rho_\theta)^*=u^*\ast \rho_\theta$. To take the limit in \eqref{e:interp_1} (and similarly in \eqref{e:interp_2}) we observe that all the terms in the right hand side pass to the limit because $u^*\ast \rho_\theta \to u^*$, $\overline \Delta_b u^*\ast \rho_\theta \to \overline \Delta_b u^*$ in $L^2(\RR^4_+, y^b)$, while for the left hand side we employ a lower-semicontinuity argument.
\end{proof}

\begin{proof}[Proof of Lemma \ref{lemma:suitable}]
	 Let $r,s \in (3/4,1)$, $r<s$, $r' = (2r+s)/3$, $s'= (r+2s)/3$ and let us consider a cutoff $\varphi$ supported in $Q^*_{s'}$, which is identically $1$ on $Q^*_{r'}$ and such that 
	 \begin{align*}
	 |\partial_t \varphi| + |\partial_y\varphi| + |\overline\nabla \varphi| &\leq \frac{C}{r-s}\,,\qquad  |\overline\nabla^2 \varphi| \leq \frac{C}{(r-s)^{2}}
	  \end{align*}
for some universal constant $C$. Moreover, we assume that $\varphi$ is constant in the variable $y$, namely $\partial_y\varphi=0$, on the domain $\{y<\frac 12\}$. In particular we conclude that $|\overline\Delta_b\varphi|\le C(r-s)^{-2}$.
Let us consider
	$$
	h(r):= \sup_{t\in [-r^{2\alpha},0]}\int_{B_r} \frac{|u(x,t)|^2}{2}\, dx + c_{\alpha}\int_{Q_r^*} y^b| \overline\Delta_b u^*|^2 \, dx \,dy\,dt'.
	$$ 
	By the energy inequality \eqref{eqn:suit-weak-drift} we have that for every $t\in [-r^{2\alpha},0]$
	\begin{align}
	& \hphantom{\le -}\int_{B_r} |u(\cdot,t)|^2 \, dx +2c_{\alpha} \int_{-r^{2\alpha}}^t \int_{B_r^*} y^b|\overline\Delta_b u^*|^2 \, dx \, dy\, dt'\nonumber\\
	& \le\int_{B_s} \varphi(\cdot,0,t) |u(\cdot,t)|^2 \, dx +2c_\alpha \int_{-s^{2\alpha}}^t \int_{\RR^4_+} y^b|\overline\Delta_b u^*|^2 \varphi \, dx \, dy\, dt'\nonumber\\ 
	& \le \hphantom{-}\int_{-s^{2\alpha}}^t \int_{\RR^3} \Big[ |u|^2 \partial_t \varphi + (L| u |^2 +2p) u \cdot \nabla \varphi \Big] \, dx \, dt'+\int_{-s^{2\alpha}}^t\int_{B_s}M\cdot\nabla\varphi|u|^2\,dx\,dt'\nonumber\\  
	&\hphantom{\le} -c_{\alpha}\int_{-s^{2\alpha}}^t\int_{\RR^4_+}y^b\left(2\overline\nabla\varphi\overline \nabla u^* \overline\Delta_b u^*+ u^* \overline\Delta_b\varphi\overline\Delta_b u^*\right) \, dx \, dy \, dt'\,.\label{eqn:suit-weak-hp2}
	\end{align}
	Regarding the non-quadratic term in the right hand side, we notice that, since $u$ is divergence-free, for any $f\in L^1([-1,0])$,
	$$
	\int_{-s^{2\alpha}}^t \int_{\RR^3} \Big[ (L| u |^2 +2p) u \cdot \nabla \varphi \Big] \, dx \, dt' =\int_{-s^{2\alpha}}^t \int_{\RR^3} \Big[ (L| u |^2 -f(t)+2p) u \cdot \nabla \varphi \Big] \, dx \, dt'.$$
	
\noindent Taking the supremum as $t$ varies in $[-r^{2\alpha},0]$ in \eqref{eqn:suit-weak-hp2} and using the bounds on the derivatives of $\varphi$, together with the fact that they vanish inside $Q_{r'}$ and outside $Q_{s'}$, we deduce that
	\begin{align}
	h(r) \leq & \hphantom{+}C
	\int_{Q_{s'}\setminus Q_{r'}} (1+|M|) \frac{|u|^2}{s-r} + \frac{|L| u |^2 -f(t)+2p | |u| }{s-r}\,dx\,dt \nonumber\\
	& + C \int_{Q_{s'}^*\setminus Q_{r'}^*}\frac{y^b|\overline\Delta_b u^*| |u^*| }{(s-r)^2}
	+ \frac{y^b|\overline\nabla u^*||\overline\Delta_b u^*|}{s-r} \, dx \,dy\, dt'.\label{eqn:suit-est}
	\end{align}
	We estimate each term in the right hand side of \eqref{eqn:suit-est}.
	For the first term, we notice that
	\begin{align*}
	\int_{Q_{s'}\setminus Q_{r'}}(1+|M|) \frac{|u|^2}{s-r}  \, dx \, dt' & \leq \frac{(1+|M|)}{s-r} \int_{Q_{1}}|u|^2   \, dx \, dt'\,.
\end{align*}
	For the third term, we apply Young inequality to deduce that
	\begin{align*}
	\int_{Q_{s'}^*\setminus Q_{r'}^*}  \frac{y^b|\overline\Delta_b u^*| |u^*| }{(s-r)^2}
	\, dx \, dy\, dt' &\leq \int_{Q_{s'}^*\setminus Q_{r'}^*} y^b|\overline\Delta_b u^*|^2  \, dx \, dy\,dt' + \frac{1 }{(s-r)^4}
	\int_{Q_{s'}^*\setminus Q_{r'}^*} y^b|u^*|^2 \, dx \, dy\, dt'\\
	& \le (h(s)-h(r)) +\frac{1 }{(s-r)^4}
	\int_{Q_{1}^*} y^b|u^*|^2 \, dx \, dy\, dt'.
	\end{align*}
	Finally, for the last term we first use Young inequality to infer
	\begin{align*}
	\int_{Q_{s'}^*\setminus Q_{r'}^*}y^b\overline\nabla\varphi\overline\nabla u^*\overline\Delta_b u^* \, dx \,dy \, dt'& \le \int_{Q_{s'}^*\setminus Q_{r'}^*} y^b|\overline\Delta_b u^*|^2\,dx\,dy\,dt' +\frac{C}{(s-r)^2}\int_{Q_{s'}^*\setminus Q_{r'}^*} y^b|\overline\nabla u^*|^2\,dx\,dy\,dt'\\
	& \le (h(s)-h(r))+\frac{C}{(s-r)^2}\int_{Q_{s'}^*\setminus Q_{r'}^*} y^b|\overline\nabla u^*|^2\,dx\,dy\,dt'.
	\end{align*} 
	In order to estimate the last integral, we decompose the domain of integration as 
\begin{equation}\label{e:decomp_1}
Q_{s'}^*\setminus Q_{r'}^* = \big( (B_{s'}^* \setminus B_{r'}^*)  \times (-(r')^{2\alpha},0] \big) \cup \big( B_{s'}^* \times (-(s')^{2\alpha},-(r')^{2\alpha}] \big) .
\end{equation}
	Next, we apply Lemma \ref{lem:interpol} twice: consider at first the case $t\in (-(r')^{2\alpha},0]$ and a cutoff function $\psi\in C^\infty_c(\RR^4)$ in the variables $x$ and $y$, with the following properties: ${\rm supp}\,\psi\cap\RR^4_+\subset B_s^*\setminus B_r^*$, $\psi\equiv 1$ in $B_{s'}^*\setminus B_{r'}^*$ and $|\overline\nabla\psi|\le\frac{C}{r-s}$.
	Thus, for each time $t\in (-r^{2\alpha},0]$ we have
	\begin{equation}\label{leones}
\frac{1}{(s-r)^2}\int_{B_{s'}^*\setminus B_{r'}^*} \!\!\!\!\!\! y^b |\overline\nabla u^*|^2\,dx\,dy\le\int_{B_s^*\setminus  B_{r}^*}  \!\!\!\!\!\! y^b|\overline\Delta_b u^*|^2	\,dx\,dy+\frac{C}{(s-r)^4}\int_{B_1^*} \!\!\! y^b|u^*|^2\,dx\,dy.
	\end{equation}
	Next, at each fixed time $t\in (-(s')^{2\alpha},-(r')^{2\alpha})$ apply Lemma \ref{lem:interpol} with  $\varepsilon=(s-r)^{-2}$ and a new cutoff function $\psi\in C^\infty_c(\RR^4)$ in the variables $x$ and $y$, with the following properties: ${\rm supp}\,\psi\cap\RR^4_+\subset B_s^*$, $\psi\equiv 1$ in $B_{s'}^*$ and $|\overline\nabla\psi|\le\frac{C}{r-s}$.
	Thus, for each time $t\in (-(s')^{2\alpha},-(r')^{2\alpha})$ we have
	\begin{equation}\label{leones2}
\frac{1}{(s-r)^2}\int_{B_{s'}^*} y^b |\overline\nabla u^*|^2\,dx\,dy\le\int_{B_s^*} y^b|\overline\Delta_b u^*|^2	\,dx\,dy+\frac{C}{(s-r)^4}\int_{B_1^*} y^b|u^*|^2\,dx\,dy.
	\end{equation}
We integrate in time \eqref{leones} for $t\in (-(r')^{2\alpha},0]$ and \eqref{leones2} for $t\in (-(s')^{2\alpha},-(r')^{2\alpha})$ and we sum the two inequalities. We then use \eqref{e:decomp_1} and
the inclusion
$$\big(  (B_{s}^* \setminus B_{r}^*) \times (-(r')^{2\alpha},0] \big) \cup \big( B_{s}^* \times (-(s')^{2\alpha},-(r')^{2\alpha}]\big) 
\subseteq Q_{s}^*\setminus Q_{r}^*.
$$
Summing the corresponding contributions we get
	\begin{align*}
	\frac{1}{(s-r)^2}\int_{Q^*_{s'}\setminus Q^*_{r'}} y^b |\overline\nabla u^*|^2\,dx\,dy\,dt'
	&\le \int_{ Q^*_{s}\setminus Q^*_{r}} y^b |\overline\nabla u^*|^2\,dx\,dy\,dt'
	+\frac{C}{(s-r)^4}\int_{Q^*_1} y^b|u^*|^2\,dx\,dy\,dt'
	\\
	&\le(h(s)-h(r))+\frac{C}{(s-r)^4}\int_{Q^*_1} y^b|u^*|^2\,dx\,dy\,dt'.
	\end{align*}
	
\noindent Putting together the previous estimates in \eqref{eqn:suit-est}, we obtain that for every $3/4 < r<s<1$
	\begin{multline*}
	h(r) \leq C_* (h(s)-h(r)) + C (1+|M|) \frac{1 }{s-r}
	\int_{Q_{1}} |u|^2 \, dx \, dt' +  \frac{1}{s-r}\int_{Q_{1}}\big|L| u |^2 -f(t)+2p \big| |u|   \, dx \, dt'\\
	+\frac{C}{(s-r)^4}\int_{Q^*_1} y^b|u^*|^2\,dx\,dy\,dt'
	\end{multline*}
	for universal constants $C_*$ and $C$.
	Hence, we apply Lemma 6.1 of \cite{Giusti} to conclude that 
	$$h(3/4) \leq C  (1+|M|)
	\int_{Q_{1}} |u|^2 \, dx \, dt' + C\int_{Q_{1}}|L| u |^2 -f(t)+2p | |u|   \, dx \, dt'+C\int_{Q^*_1} y^b|u^*|^2\,dx\,dy\,dt',$$
	which proves \eqref{eqn:loc-en}.
\end{proof}

We close the section
with the following lemma where we estimate the nonlocal term in the right hand side of \eqref{eqn:loc-en} in terms of the $L^2$-norm of $u$ and of its ``tail''.
\begin{lemma}\label{lem:withtails}
Let $\alpha\in (1,3/2]$ and $u\in L^2(B_1)\cap L^1_{\rm loc}(\RR^3)$ and assume that
\[
\sup_{R\ge 1}R^{-3\alpha}\left(\mean{B_R}|u|\,dx\right)^2 < +\infty\, .
\]
Consider the kernel $P$ of Proposition \ref{prop:poisson_k}. Then the associated extension $u^* (x,y)= (P (\cdot, y)\ast u) (x)$  belongs to $L^2_{\rm loc} (\RR^4_+, y^b)$ and satisfies the estimate
\begin{equation}\label{e:extu}
\int_{B_1^*}|u^*|^2 y^b\,dx\,dy\le C\left(\int_{B_2} |u|^2+\sup_{R\ge 1}R^{-3\alpha}\left(\mean{B_R}|u|\,dx\right)^2\right),
\end{equation}
for a geometric constant $C$ (which is, in particular, independent of $\alpha$).
\end{lemma}

\begin{proof}
We set $u_1:={\mathbbm 1}_{B_2} u$ and $u_i:=\left({\mathbbm 1}_{B_{2^{i+1}}}-{\mathbbm 1}_{B_{2^{i}}}\right)u$ for $i>1$.
Thanks to the computation of the Poisson kernel $P(x,y)$ of $\overline\Delta_b^2$ in Proposition \ref{prop:poisson_k} and to \eqref{e:rappresentazione} we can write
\[
u^*(x,y)=u\ast P(\cdot,y)=\sum_{i=1}^\infty u_i\ast P(\cdot, y)=\sum_{i=1}^\infty u_i^*\,,
\]
where $P(x,y)={y^{2\alpha}}{(|x|^2+y^2)^{-\frac{3+2\alpha}{2}}}$.
Thus
\begin{equation}\label{e:shells}
\int_{B_1^*}|u^*|^2y^{b}\,dx\,dy\le 2\left(\int_{B_1^*}|u_1^*|^2 y^b \,dx\,dy+\int_0^1y^b\left(\sum_{i>1} \|u_i^*(\cdot, y)\|_{L^\infty(B_1)}\right)^2\,dy\right).
\end{equation}
To estimate the first term in the right hand side of \eqref{e:shells}, we observe that $P(x,y)=y^{-3}P(y^{-1}x,1)$, thus
\[
\widehat P(\xi,y)= y^{-3}\widehat{P\left(y^{-1}\cdot,1\right)}(\xi)=\widehat{P(\cdot,1)}(\xi y)\le\|\widehat{P(\cdot,1)}\|_{L^\infty}<\infty\,.
\]
Combining the latter with Plancherel's identity we achieve
\begin{align*}
\int_{B_1^*}|u_1^*|^2y^b\,dx\,dy & \le \int_{\RR^3\times[0,1]} |u_1^*|^2y^b\,dx\,dy =\int_{\RR^3\times[0,1]}|\widehat{u_1^*}|^2y^b\,d\xi\,dy =\int_{\RR^3\times[0,1]}|\widehat u_1(\xi)\widehat{P}(\xi,y)|^2y^{-1-2\alpha}\,d\xi\,dy\\
& \le \|\widehat{P(\cdot,1)}\|^2_{L^\infty}\int_{\RR^3}|\widehat{u}_1(\xi)|^2\,d\xi\int_0^1 y^{b}\,dy \le C\int_{\RR^3}|u_1|^2\,dx=C\int_{B_2}|u_1|^2\,.
\end{align*}
Moreover, for $i>1$,
\[
u_i^* (x,y) =( u_i\ast P(\cdot,y)) (x)=\int_{\RR^3}u_i(z)P(x-z,y)\,dz=\int_{B_{2^{i+1}}\setminus B_{2^{i}}} u(z)P(x-z,y)\,dz\,.
\]
When $z\in B_{2^{i+1}}\setminus B_{2^i}$ we have that $|x-z|>2^{i-1}$ and $P(x-z,y)\le P(2^{i-1},y)$, thus
\[
\|u_i^*(\cdot, y)\|_{L^{\infty}(B_1)}\le P(2^{i-1},y)\int_{B_{2^{i+1}}\setminus B_{2^i}}|u|\le\int_{B_{2^{i+1}}\setminus B_{2^{i}}} \frac{|u|y^{2\alpha}}{2^{(3+2\alpha)(i-1)}}\,.
\]
Adding over $i$, we obtain
\begin{align*}
\sum_{i>1}\|u_i^*(\cdot,y)\|_{L^\infty(B_1)} & \le \sum_{i>1}\int_{B_{2^{i+1}}\setminus B_{2^{i}}}\frac{|u|y^{2\alpha}}{2^{(3+2\alpha)(i-1)}}\le\sum_{i>1}\frac{1}{2^{\frac\alpha 2(i-1)}}\int_{B_{2^{i+1}}\setminus B_{2^{i}}} \frac{|u|y^{2\alpha}}{2^{\left(3+\frac 32\alpha\right)(i-1)}}\\
& \le \left(\sum_{i>1}\frac{y^{2\alpha}}{2^{\frac \alpha 2(i-1)}}\right)\sup_{R\ge 1} \int_{B_R}\frac{|u|}{R^{3+\frac 32\alpha}}\le Cy^{2\alpha}\sup_{R\ge 1} \int_{B_R}\frac{|u|}{R^{3+\frac 32\alpha}} \,.
\end{align*}
Replacing the right hand side of \eqref{e:shells} with the latter estimates, we conclude \eqref{e:extu}.
\end{proof}

\section{Compactness of suitable weak solutions}\label{s:compactness}

We now use the energy estimates of the previous section to prove a compactness statement for suitable rescalings and normalizations of solutions. A crucial assumption of such compactness lemma is that we have an appropriate control of some local quantities. For this reason we introduce a notion of ``excess'' (for which in the next section we will prove a suitable decay property).

\begin{definition}{\rm 
We define the excess as $E (u,p;x,t,r)= E^V (u;x,t,r) + E^P (p;x,t,r) + E^{nl} (u;x,t,r)$ where
\begin{align}
E^V(u;x,t,r) &:= \left(\mean{Q_r (x,t)}|u-(u)_{Q_r(x,t)}|^3\right)^\frac 13\\ 
E^P(p;x,t,r) &:= r^{2\alpha-1}\left(\mean{Q_r (x,t)} |p-[p]_{B_r (x)}|^\frac 32\right)^\frac 23\\
E^{nl} (u;x,t,r) &:= \left(\mean{-r^{2\alpha}}^0 \sup_{R\ge\fueta r}\left(\frac{r}{R}\right)^{3\alpha} \mean{B_R (x)}|u- (u)_{Q_r (x,t)}|^2\right)^\frac 12\, .
\end{align}
}\end{definition}

We observe that, if $u$ is a solution of \eqref{e:NS_alfa}, then
\begin{equation}\label{e:scaling}
u_r(x,t) :=r^{2\alpha-1}u(rx,r^{2\alpha}t), \qquad
p_r(x,t)  := r^{4\alpha-2}p(rx, r^{2\alpha}t) 
\end{equation}
is a solution, too. The rescaling of the excess is given by
\begin{equation}\label{e:scaling2}
E(u_r,p_r;0,0,1)=r^{2\alpha-1}E(u,p;0,0,r).
\end{equation}
We are now ready to state our compactness property. 

\begin{lemma}\label{lem:lions}
Let $(u_k,p_k)$ be a sequence of suitable weak solutions of the Navier-Stokes system \eqref{e:NS_alfa} in $\RR^3 \times [-1,0]$, with $(u_k)_{Q_1}\to M_0$ and $E(u_k,p_k;0,0,1)\to 0$. Define the rescalings 
\begin{equation*}
v_k :=\frac{u_k-(u_k)_{Q_1}}{E(u_k,p_k;0,0,1)}\qquad q_k :=\frac{p_k-[p_k]_{B_1}}{E(u_k,p_k;0,0,1)}\,.
\end{equation*}
Up to subsequences $(u_k, p_k)$ converges (in the sense of distributions) to a
pair $(v,q)$ defined on $\RR^3\times[-1,0]$ which solves the following system in $Q_\frac 12$
\begin{equation}\label{e:linNS_1}
\left\{
\begin{array}{l}
\partial_t v+\nabla q+ M_0\cdot\nabla v=-(-\Delta)^\alpha v \\ \\
{\rm div}\, v=0\, 
\end{array}
\right.
\end{equation}
and satisfies
\begin{equation}\label{e:ci_vuole!}
E (v;0,0,1) \leq 1\,.
\end{equation}
Moreover, $v_k\to v$ strongly in $L^3\big(Q_\frac 12,\RR^3\big)$ and $q_k\to q$ strongly in $L^{\frac 32}\big(Q_\frac 12\big)$.
\end{lemma}

\begin{remark}\label{r:cosa_v_e_q}
Note that $v$ belongs to $L^2_{\rm loc}$ because of \eqref{e:ci_vuole!}, whereas $q$ is, a priori, just a distribution. Note that, although the behavior of $q$ outside $Q_1$ does not affect the equation \eqref{e:linNS_1}, the global behavior of $v$ enters in \eqref{e:linNS_1} since it affects the nonlocal operator $(-\Delta)^\alpha$. 
\end{remark}

An important ingredient of the proof is the following interpolation lemma. 

\begin{lemma}\label{lem:cpt_fixedt}
If $v\in L^2(\RR^3;\RR^3)$ then for every $r\in (0,1)$ there exists $C:=C(r)>0$ such that
\begin{equation}\label{e:bound4cpt}
\|v\|^2_{L^{\frac{6}{3-2\alpha}}(B_{r})} \leq C \Big(\int_{B_1^*}y^b|\overline\Delta_b v^*|^2+ \sup_{R\ge\fueta}R^{-3\alpha}\mean{B_R}|v|^2 \Big).
\end{equation}
\end{lemma}
\begin{proof}
Fix $r\in (0,1)$ and $\varphi$ be a smooth cutoff function between $B_{r}^*$ and $B_1^*$. We assume that the cutoff function $\varphi|_{\{y<\frac 12\}}$ is independent from the variable $y$, so that $\overline\nabla\varphi=\nabla\varphi$ in the set $\{y<\frac 12\}$. We can estimate
\begin{align*}
\int_{\RR^3}|(-\Delta)^{\frac{\alpha}{2}}(v\varphi)|^2
& \stackrel{\eqref{e:en_of_ext}}{=}c_{\alpha}\int_{\RR^4_+}y^b|\overline\Delta_b(v\varphi)^*|^2\stackrel{\eqref{e:minimum_yang_ext}}{\le}c_{\alpha}\int_{\RR^4_+}y^b|\overline\Delta_b(v^*\varphi)|^2\\
& \ \le C\int_{B_1^*}y^b \big(|\overline\Delta_bv^*|^2 \varphi^2+ |\overline\nabla v^*|^2|\overline\nabla \varphi|^2+ |\overline\Delta_b\varphi|^2 |v^*|^2 \big).
\end{align*}
The second summand of the right hand side can be estimated by Lemma \ref{lem:interpol} with $\varepsilon=1/2$, $\psi=|\overline\nabla\varphi|$
$$
\int_{\RR^4_+}y^b|\overline\Delta_b(v^*\varphi)|^2\le C\int_{B_1^*}y^b \big(|\overline\Delta_bv^*|^2 + |v^*|^2 \big)\,,
$$
which in turn can be estimated with the right hand side of \eqref{e:bound4cpt} thanks to Lemma \ref{lem:withtails}. Next, by Sobolev's embedding
\[
\|v\|^2_{L^{\frac{6}{3-2\alpha}}(B_{r})}\le \|v\varphi\|^2_{L^{\frac{6}{3-2\alpha}}(\RR^3)}\le C\int_{\RR^3}|(-\Delta)^{\frac{\alpha}{2}}(v\varphi)|^2\le C \Big(\int_{B_1^*}y^b|\overline\Delta_b v^*|^2+ \sup_{R\ge\fueta}R^{-3\alpha}\mean{B_R}|v|^2 \Big).\qedhere
\]
\end{proof}

\begin{proof}[Proof of Lemma \ref{lem:lions}]
	By Lemma~\ref{lemma:drift}, the pair $(v_k,q_k)$ is a suitable weak solution of the modified Navier-Stokes equation \eqref{e:sws-const} with $L= E(u_k,p_k;0,0,1)$ and $M= (u_k)_{Q_1}$. By Lemma~\ref{lemma:suitable}, we obtain 
	\begin{equation}\label{e:swe_lin}
	\sup_{t\in \left[-\left(\frac 34\right)^{2\alpha},0\right]}\int_{B_\frac 34}\frac{| v_k(\cdot,t)|^2}{2}\, dx + \int_{Q^*_\frac 34} y^b | \overline\Delta_b v_k^*|^2 \, dx\, dy \, dt \leq C\,,
	\end{equation}
	where $C$ does not depend on $k$.	
From \eqref{e:swe_lin} and Lemma \ref{lem:cpt_fixedt} we obtain that the sequence $(v_k)$ is bounded in $L^\infty \big([-\left(\frac 58\right)^{2\alpha},0],L^2(B_{\frac 58})\big)$ and in $L^2\big([-\left(\frac 58\right)^{2\alpha},0],L^{\frac{6}{3-2\alpha}}(B_{\frac 58})\big)$. By interpolation, $v_k$ is uniformly bounded in $L^{\frac{10}{3}}\big(Q_\frac 58\big)$.

Without loss of generality we can extract a subsequence (not relabeled) converging weakly to some $v$ in $L^2_{\rm loc}(\RR^3\times (-1,0])$, notice in fact that the nonlocal part of the excess controls uniformly the $L^2$-norm in $B_r\times (-1,0]$ for any $r>0$. Similarly, using the equation $\Delta p_k = {\rm div}\, {\rm div}\, (v_k\otimes v_k)$, we can assume that the sequence $\{p_k\}$ is suitably bounded in the space of tempered distributions and converges to some distribution $q$.
Note moreover that \eqref{e:ci_vuole!} follows from the lower semicontinuity of $E$. 

We claim that the sequence $v_k$ is a Cauchy sequence in $L^1\big(Q_\frac 12\big)$. Indeed let $\varepsilon>0$ be a fixed positive number and $\varphi_\theta$ be a smooth family of mollifiers in space supported in $B_1$. As shown in the proof of Lemma \ref{lem:cpt_fixedt}, the sequence $(v_k)$ is bounded in $L^2 ((-1, 0], H^\alpha (B_{5/8}))$. In particular, by the compact embedding of $H^\alpha (B_{5/8})$ in $L^2 (B_{5/8})$ and the uniform bound in $L^{\frac{10}{3}}\big(Q_\frac 58\big)$, we easily conclude
\[
\|v_k-v_k\ast\varphi_\theta\|_{L^3(Q_{1/2})}\le  f(\theta)
\]
for some function $f$ of $\theta$ which is independent of $k$ and which converges to $0$ as $\theta\downarrow 0$.
Therefore there exists $\theta>0$ such that
\begin{equation}\label{e:cauchy_conv}
\|v_k-v_k\ast\varphi_\theta\|_{L^3(Q_{1/2})}\le\frac{\varepsilon}{3}\quad \forall\,k\ge 1.
\end{equation}
We analyze the equation for $v_k\ast\varphi_\delta$:
\[
\partial_t(v_k\ast\varphi_\theta)=-\nabla p_k\ast\varphi_\theta -[(v_k\cdot\nabla)v_k]\ast\varphi_\theta-v_k\ast(-\Delta^\alpha)\varphi_\theta\,. 
\]
Observe that, using $\nabla p_k\ast\varphi_\theta=p_k\ast\nabla\varphi_\theta$ and the uniform bound in $L^{\frac 32}$ for $p_k$, we have
\[
\|\nabla p_k\ast\varphi_\theta\|_{L^{\frac 32}(]-\left(1/2\right)^{2\alpha},0]; W^{1,\infty}(B_{1/2}))}\le C(\theta)\,,
\]
where $C(\theta)$ is a constant which possibly blows up when $\theta\to 0$, but otherwise does not depend on $k$.
Using $[(v_k\cdot\nabla)v_k] \ast\varphi_\theta={\rm div}\,(v_k\otimes v_k)\ast\varphi_\theta$ we obtain
\[
\|(v_k\cdot\nabla) v_k\ast\varphi_\theta\|_{L^\frac 32(]-\left(1/2\right)^{2\alpha},0],W^{1,\infty}(B_{1/2}))}\le C(\theta).
\]
We observe that $(-\Delta)^\alpha\varphi_\theta(x)\le\frac{C(\theta)}{1+|x|^{3+2\alpha}}$ and a similar estimate holds for the gradient of $(-\Delta)^\alpha\varphi_\theta$; indeed, since $\varphi_\theta$ is smooth, for $|x|<1$ we can use the representation formula for the fractional Laplacian and an integrating by parts to compute
\[
(-\Delta)^\alpha\varphi_\theta(x)=(-\Delta)^{\alpha-1}\Delta\varphi_\theta(x)=-C\int_{\RR^3}\frac{\Delta\varphi_\theta(y)}{|x-y|^{3+2(\alpha-1)}}\,dy=C\int_{\RR^3}\frac{\varphi_\theta(y)}{|x-y|^{3+2\alpha}}\,dy\le\frac{C(\theta)}{1+|x|^{3+2\alpha}}\,.
\] 
Recalling also that $E^{nl}(v_k;1)\le 1$ and estimating the convolution on each dyadic annulus, we obtain
\[
\|v_k\ast(-\Delta)^\alpha\varphi_\theta\|_{L^{2}(]-(1/2)^{2\alpha},0],W^{1,\infty}(B_{1/2}))}\le C(\theta).
\]
The sequence of maps $t\mapsto v_k\ast \varphi_\theta(\cdot,t)$ is therefore a sequence of equicontinuous maps taking values in a closed bounded subset $X\subset W^{1,\infty}(B_{1/2})$. Endowing $X$ with the uniform metric we have a compact metric space. From Ascoli-Arzel\`a theorem the sequence is precompact. On the other hand we know it is converging weakly in $L^{\frac{10}{3}}(Q_{1/2})$ to $v\ast\varphi_\theta$, thus $v_k\ast\varphi_\theta$ is converging uniformly to $v\ast\varphi_\theta$. Hence there exists a $K\in \NN$ such that for every $j,k\ge K$ we have
\[
\|v_k\ast\varphi_\theta-v_j\ast\varphi_\theta\|_{L^1(Q_{1/2})}<\frac{\varepsilon}{3}.
\] 
This, together with \eqref{e:cauchy_conv}, completes the proof that $v_k$ is a Cauchy sequence in $L^1 (Q_{1/2})$. The strong convergence of $p_k$ can be inferred from Calder\'on-Zygmund estimates.
\end{proof}

\section{Estimates on the linearized equation}\label{s:linear}

This section is entirely devoted to prove the following linear estimate.

\begin{lemma}\label{lem:hoelder}
Let $\alpha \in (1,2)$, $\delta\ge 0$ and let {$u\in L^2_{\rm loc} (\RR^3\times [0,1] ; \RR^3)$ and $p\in L^{3/2} (Q_1, \RR)$} be solutions in $Q_1$ of the linearized system \eqref{e:linNS_1}
with $M_0\in\RR^3$, $|M_0|\le M$, $(u)_{Q_1}=0$, $[p]_{B_1}=0$ {and
$E^{nl} (u;0,0,1) < \infty$.
} 

Then there exists a constant $\overline C=\overline C(M,\alpha)$ such that
\begin{equation}
\label{ts:linearizzato}
[u]_{C^{\frac 13}(Q_{1/2})}+[p]_{L^{\frac 32}\left(\left(-\frac 12,0\right);C^{1/2} (B_{1/2})\right)}
\le \overline C\left(\int_{Q_1}|u|^3\right)^\frac 13+\overline C\left(\int_{Q_1}|p|^\frac 32\right)^\frac 23 +{
\overline C \left(\int_{-1}^0 \sup_{R\ge\fueta }{R}^{-3\alpha} \mean{B_R (x)}|u|^2\right)^\frac 12}
.
\end{equation}
Moreover the constant $\overline C$ is independent of $\alpha$, when the latter varies in $(6/5,5/4]$.
\end{lemma}

\begin{proof}
By a standard regularization argument (which can be used because \eqref{e:linNS_1} is a linear system with constant coefficients) we assume $u, p\in C^\infty$. In this proof we denote by $C$ a generic constant possibly depending on $\alpha$ and $M_0$, and by $u^*$ the map $P (\cdot, y)\ast u$ with $P$ as in Proposition \ref{prop:poisson_k}. 
In the proof we use several cut-off functions on $\RR^4_+$. We follow the convention that a ``cut-off function between $A$ and $B$'' denotes a smooth function equal to $1$ on $A$ and compactly supported in $B$. Such cut-offs might be nonzero on $\{y=0\}$ and are assumed to be constant in the variable $y$ over a suitable strip $\RR^3\times [0, y_0]$.

We derive a first energy estimate multiplying the equation by $\varphi_1^2 u$, where $\varphi_1$ is a cutoff between $B_{7/8}^*$ and $B_1^*$
\[
\frac{d}{dt}\int_{B_1}|u|^2\varphi_1^2=2\int_{B_1} \left(p\nabla\varphi_1\cdot u\varphi_1+ M_0\cdot\nabla\varphi_1|u|^2 \varphi_1-(-\Delta)^\alpha u\cdot u\varphi_1^2\right)\, .
\]
Arguing as in the proof of Lemma \ref{lemma:suitable}, with the obvious adjustments, we conclude that
\begin{align*}
& \sup_{t\in\left(-(\frac 34)^{2\alpha},0\right]}\frac 12\int_{B_{3/4}} |u|^2+\int_{-\left(\frac 34\right)^{2\alpha}}^0\int_{B_{3/4}^*} |y|^b|\overline\Delta_b u^*|^2\\
\le  & C \left(\int_{Q_1} |p|^\frac 32\right)^{\frac 23}\left(\int_{Q_1} |u|^3\right)^\frac 13(1+M)+CM\int_{Q_1}|u|^2+C\int_{Q_{1}^*}y^b|u^*|^2.
\end{align*}
Consider now a new cutoff function $\varphi_{3/4}$ between $B_{1/2}^*$ and $B_{3/4}^*$.  Thanks to fractional Sobolev embeddings, the properties of the extension $u^*$ \eqref{e:en_of_ext} and \eqref{e:minimum_yang_ext} and the interpolation result of Lemma \ref{lem:interpol}, we have:
\begin{align}
\int_{Q_\frac 34} |\nabla u|^2 & \le \int_{-\left(\frac 34\right)^{2\alpha}}^{0}\int_{\RR^3}|\nabla(u\varphi_{3/4})|^2\le C\int_{-\left(\frac 34\right)^{2\alpha}}^{0}\int_{\RR^3}|(-\Delta)^{\frac{\alpha}{2}}(u\varphi_{3/4})|^2=\int_{-\left(\frac 34\right)^{2\alpha}}^{0}\int_{\RR^4_+} y^b|\overline\Delta_b (u\varphi_{3/4})^*|^2\nonumber\\
& \le\int_{-\left(\frac 34\right)^{2\alpha}}^{0}\int_{\RR^4_+} y^b|\overline\Delta_b (u^*\varphi_{3/4})|^2\nonumber\\
&\le\int_{-\left(\frac 34\right)^{2\alpha}}^{0}\int_{\RR^4_+} y^b\varphi_{3/4}^2|\overline\Delta_b u^*|^2+C\int_{-\left(\frac 34\right)^{2\alpha}}^{0}\int_{\RR^4_+}y^b|\overline\nabla\varphi_{3/4}|^2|\overline\nabla u^*|^2+C\int_{Q_1^*}y^b|u^*|^2\nonumber\\
&\le C\int_{-\left(\frac 34\right)^{2\alpha}}^{0}\int_{B_1^*} y^b |\overline\Delta_b u^*|^2\varphi_{3/4}^2+C\int_{Q_1^*}y^b|u^*|^2.
\end{align}
Concerning the pressure, let $\psi\in C^\infty_c(B_{3/4})$ be a further cutoff function, identically $1$ on $B_{23/32}$. Let $\hat p$ be the potential theoretic solution of
\[
\Delta\hat p={\rm div}(M_0\cdot\nabla(\psi u)).
\]
By Calder\'on-Zygmund estimates
\[
\|\nabla\hat p\|_{L^\frac 32_tL^2_x(Q_{{11}/{16}})}\le\|M_0\cdot\nabla(\psi u)\|_{L^\frac 32_tL^2_x(Q_{{11}/{16}})}\le C \|\nabla u\|_{L^\frac 32_tL^2_x(Q_{3/4})} +C\|u\|_{L^\frac 32_tL^2_x(Q_{3/4})}.
\]
Since $\hat p-p$ is harmonic in $B_{{23}/{32}}$ we can control
\[
\|\nabla(\hat p -p)\|_{L^\frac 32_t L^2_x(B_{{11}/{16}})}\le C\|\hat p-p\|_{L^\frac 32_{x,t}(B_{{23}/{32}})}\le C\Big(\|p\|_{L^\frac 32_{x,t}(B_{3/4})}+ \|u\|_{L^2_{x,t}(B_{3/4})}\Big).
\]
Therefore
\begin{equation}
\label{eqn:p-aiuto}
\|\nabla p\|_{L^{\frac 32}_t L^2_x(Q_{{11}/{16}})}\le \|\nabla u\|_{L^2(Q_{3/4})}+C\|u\|_{L^2(Q_1)}+C\|p\|_{L^{\frac 32}(Q_1)}.
\end{equation}

Observe that, by linearity again, any derivative $\partial_i u$ is a solution of \eqref{e:linNS_1}, in particular
\[
\partial_t \partial_i u+\nabla\partial_i p+ M_0\cdot\nabla \partial_i u=-(-\Delta)^\alpha \partial_i u.
\]
Let $\varphi_{11/16}$ be a cutoff function between $B_{5/8}^*$ and  $B_{11/16}^*$. Multiplying the equation by $\varphi_{11/16}\partial_i u$ and integrating in space, we obtain
\[
\frac{d}{dt}\int_{B_1}|\partial_i u|^2\varphi_{11/16}^2=2\int_{B_1} \partial_i p\nabla\varphi_{11/16}\cdot \partial_i u\varphi_{11/16}+ M_0\cdot\nabla\varphi_{11/16}|\partial_i u|^2 \varphi_{11/16}-(-\Delta)^\alpha \partial_i u\cdot \partial_i u\varphi_{11/16}^2.
\]

\noindent Integrating in time and performing tricks as above,  
\begin{multline}
\label{eqn:energy-ineq-deriv-lin}
\sup_{t\in\left(-(\frac{11}{16})^{2\alpha},0\right]} \int_{B_1}|\partial_i u|^2\varphi_{11/16}^2 
+\int_{-\left(\frac{11}{16}\right)^{2\alpha}}^0\int y^b|\overline\Delta_b \partial_iu^*|^2\varphi_{11/16}^2
\\ \leq
2\int_{Q_{11/16}} \partial_i p\nabla\varphi_{11/16}\cdot \partial_i u\varphi_{11/16}
+C \int_{Q_{11/16}} |\partial_i u|^2 \varphi_{11/16}
+
C\int_{Q_{11/16}^*}y^b|\partial_iu^*|^2.
\end{multline}
The third term in the right hand side can be estimated via Lemma~\ref{lem:interpol}. Regarding the first term, we have:
$$
\int_{Q_{11/16}} \partial_i p\nabla\varphi_{11/16}\cdot \partial_i u\varphi_{11/16}
\leq C\| \nabla p\|^2_{L^{1} ((-(11/16)^{2\alpha},0];L^2(B_{11/16}))} + \frac{1}{4} \| \partial_i u \varphi_{11/16}\|^2_{L^\infty((-(11/16)^{2\alpha},0]; L^2(B_{11/16}))}. 
$$
The first term in the right hand side is estimated by \eqref{eqn:p-aiuto} and the second term is absorbed in the left hand side of \eqref{eqn:energy-ineq-deriv-lin}. We conclude that
\begin{equation*}
\|u\|_{L^\infty((-(\frac 58)^{2\alpha},0];H^1(B_{5/8}))} 
+\int_{Q_{5/8}} y^b|\overline\Delta_b \partial_iu^*|^2\leq
C\|u\|_{L^3(Q_1)}^2+ C\|p\|_{L^{\frac 32}(Q_1)}
+C\int_{Q_1^*}y^b|u^*|^2.
\end{equation*}

Iterating $k$ times the above estimates (possibly introducing a suitable number of intermediate radii depending on $k$)
\begin{equation}
\label{eqn:energy-ineq-,moltederiv}
\|u\|_{L^\infty((-(9/16)^{2\alpha},0],H^k (B_{9/16}))} 
\leq
C\|u\|_{L^3(Q_1)}+ C\|p\|_{L^{\frac 32}(Q_1)}
+C\int_{{Q_1}^*}y^b|u^*|^2,
\end{equation}
where the constant $C$ depends also on $k$.
Similarly, with the same argument for \eqref{eqn:p-aiuto} we achieve
\begin{equation}
\label{eqn:moltederiv-p}
\|p\|_{L^{\frac 32}((-(9/16)^{2\alpha},0], H^k(B_{9/16}))} 
\leq
C\|u\|_{L^3(Q_1)}+ C\|p\|_{L^{\frac 32}(Q_1)}
+C\int_{Q_1^*}y^b|u^*|^2.
\end{equation}

Using $k=6$ and Morrey's embedding theorem, from \eqref{eqn:moltederiv-p} we deduce the following spatial estimate for $u$:
\begin{align*}
\|u\|_{L^\infty\left((-(9/16)^{2\alpha},0], C^{3}(B_{9/16})\right)} \leq \|u\|_{L^\infty\left((-(9/16)^{2\alpha},0], H^6 (B_{9/16})\right)} \leq
C\|u\|_{L^3(Q_1)}+ C\|p\|_{L^{\frac 32}(Q_1)}
+C\int_{Q_{1}^*}y^b|u^*|^2.
\end{align*}
Thanks to Lemma \ref{lem:withtails}, the right hand side in the previous equation is estimated by the right hand side in the estimate for the pressure in \eqref{ts:linearizzato}.
From the equation \eqref{e:linNS_1}, we estimate the time derivative of $u$ with
$$
\| \partial_t u\|_{L^{\frac 32}\left((-(1/2)^{2\alpha},0],L^\infty(B_{1/2})\right)} \leq \| \nabla p \|_{L^{\frac 32}\left((-(1/2)^{2\alpha},0],L^\infty(B_{1/2})\right)}+(1+M)\|Du\|_{L^\infty(Q_1)} + \|(-\Delta)^\alpha u\|_{ L^\infty(Q_{1/2})}.
$$
For the last term, we consider a cutoff $\varphi_{9/16}$ between $Q_{1/2}$ and $Q_{9/16}$ to deduce
$$
 \|(-\Delta)^\alpha u\|_{ L^\infty(Q_{1/2})} \leq  \|(-\Delta)^\alpha (u \varphi_{9/16})\|_{ L^\infty(Q_{1/2})} + \|(-\Delta)^\alpha (u (1-\varphi_{9/16}))\|_{ L^\infty(Q_{1/2})}.
$$
The term $ \|(-\Delta)^\alpha (u \varphi_{9/16})\|_{ L^\infty(Q_{1/2})}$ is estimated by $ \|u\|_{L^\infty\left((-(1/2)^{2\alpha},0],C^3(B_{1/2})\right)} $. For the other term we observe that, for any $z \in {\rm supp }\, \varphi_{9/16}$
$$
(-\Delta)^\alpha (u (1-\varphi_{9/16}))(x) 
= \int \frac{-\Delta(u(1-\varphi_{9/16}))(z)}{|x-z|^{3+2\alpha-2}}\,dz
=c_\alpha \int \frac{(1-\varphi_{9/16}(z))u(z)}{|x-z|^{3+2\alpha}}
\, dz\,.
$$
Hence, the term $\|(-\Delta)^\alpha (u (1-\varphi_{9/16}))\|_{ L^\infty(Q_{1/2})} $ {can be estimated with the third summand in the right hand side of \eqref{ts:linearizzato}. Summarizing, we can bound $\| \partial_t u\|_{L^{\frac 32}\left((-(1/2)^{2\alpha},0],L^\infty(B_{1/2})\right)}$ with the right hand side of \eqref{ts:linearizzato}. Since
$\|u\|_{C^{\frac 13}([-(1/2)^{2\alpha},0]; L^\infty(B_{1/2}))}$ is bounded by $\| \partial_t u\|_{L^{\frac 32}\left((-(1/2)^{2\alpha},0],L^\infty(B_{1/2})\right)}$, we conclude the H\"older continuity estimate in time, which completes the proof. }
\end{proof}

\section{Excess decay and proof of the $\eps$-regularity Theorem \ref{thm:eps_reg-variant}}\label{s:excess_decay}

We now come to the core of the proof of Theorem \ref{thm:eps_reg-variant}, which is the following ``excess decay estimate''. After proving it we will show how to set up an iteration which leads to Theorem \ref{thm:eps_reg-variant}. 

\begin{propos}\label{prop:exc_decay}
Let $(u,p)$ be a suitable weak solution of the hyperdissipative Navier-Stokes equation, $\alpha \in (1, \frac{5}{4}]$ and $M\geq 0$. Then there exist $\vartheta:= \vartheta(\alpha, M)\in (0,\frac{1}{2})$ and $\eps:=\eps(\alpha, M) \in (0,\frac{1}{2})$ with the following property. If $r\leq 1$, $Q_r(x,t) \subseteq \RR^3 \times [0,T]$ for some $(x,t)\in \RR^3 \times [0,T]$ and
\[
| (u)_{Q_r(x,t)}|\le M\quad\text{ and }\quad E(u,p;x,t,r)\le \eps\, ,
\]
then
\[
E(u,p;x,t,\vartheta r)\le\frac 12 E(u,p;x,t,r).
\]
Moreover $\vartheta$ and $\varepsilon$ are uniformly bounded away from $0$ if the pair $(\alpha, M)$ ranges in a compact subset of $(1, \frac{5}{4}]\times [0, \infty)$.
\end{propos}
\begin{proof}
Without loss of generality we can assume that $(x,t)= (0,0)$ and we omit to specify it in the excess (and its variants). It suffices to prove
\[
E(u,p;\vartheta)\le\frac 12 E(u,p;1).
\]
Indeed, using \eqref{e:scaling} and \eqref{e:scaling2}, 
\[
E(u,p;r\vartheta)=r^{1-2\alpha}E(u_r,p_r,\vartheta)\le\frac{r^{1-2\alpha}}{2}E(u_r,p_r;1)=E(u,p;r)
\]
and, if $E(u,p;r)\le \varepsilon$, then in particular $E(u_r,p_r;1)\le r^{2\alpha-1}\varepsilon\le\varepsilon$.
Let us assume by contradiction that there exists a sequence $(u_k,p_k)$ such that
\begin{align}
E(u_k,p_k;\vartheta) & \ge \frac 12 E(u_k,p_k;1)\\
\lim_{k\to\infty}E(u_k,p_k;1) & =0\\
|(u_k)_{Q_1}| & \le M,
\end{align}
in particular, without loss of generality, we can assume $(u_k)_{Q_1}\to M_0$ with $|M_0|\le M$.
We set
\begin{equation*}
v_k:=\frac{u_k-(u_k)_{Q_1}}{E(u_k,p_k;1)}\quad\text{and}\quad q_k:=\frac{p_k-[p_k]_{B_1}}{E(u_k,p_k;1)}\,.
\end{equation*}
Thus $E(v_k,q_k;1)=1$ and
\begin{equation}\label{e:contr_risc}
E(v_k,q_k;\vartheta)\ge \frac 12\,.
\end{equation}
We can estimate
\begin{align*}
E^{nl}(v_k;\vartheta) & = \left(\mean{-\vartheta^{2\alpha}}^0 \sup_{R\ge \fueta \vartheta}\left(\frac{\vartheta}
{R}\right)^{3\alpha}\mean{B_R}|v_k-(v_k)_\vartheta|^{2}\right)^\frac 12\\
    & \le \left(\mean{-\vartheta^{2\alpha}}^0 \bigg(\sup_{\fueta\vartheta\le R< \fueta}\left(\frac{\vartheta}
{R}\right)^{3\alpha} \mean{B_R}|v_k-(v_k)_\vartheta|^2
+ \sup_{R\ge \fueta}\left(\frac{\vartheta}
{R}\right)^{3\alpha}\mean{B_R}|v_k-(v_k)_\vartheta|^2\bigg)\right)^\frac 12.
\end{align*}
We estimate the second term by adding and subtracting the average on $Q_1$ and
\begin{align*}
\left(\mean{-\vartheta^{2\alpha}}^0 \sup_{R\ge \fueta}\left(\frac{\vartheta}
{R}\right)^{3\alpha} \mean{B_R}|v_k-(v_k)_\vartheta|^2\right)^\frac 12
& \leq \vartheta^{\frac{\alpha}{2}}E(v_k,q_k;1)+ \left(\fuetainv\vartheta\right)^{\frac{3\alpha}{2}}
|(v_k)_{Q_\vartheta}- (v_k)_{Q_1}|
\\
& \leq \vartheta^{\frac{\alpha}{2}}+ \left(\fuetainv\vartheta\right)^{\frac{3\alpha}{2}}\left(|(v_k)_{Q_\vartheta}- (v_k)_{Q_{1/2}}| + |(v_k)_{Q_{1/2}}- (v_k)_{Q_1}|\right)
\end{align*}
and we notice that $|(v_k)_{Q_{1/2}}- (v_k)_{Q_1}| \leq c E^V(v_k;1)\le c$, so that we get

\begin{align}\label{e:nl_ecc_est}
E^{nl}(v_k;\vartheta)\le &\left(\mean{-\vartheta^{2\alpha}}^0\sup_{\fueta\vartheta\le R<\fueta}\!\left(\frac{\vartheta}{R}\right)^{3\alpha}\!\frac{1}{R^{3}}\!\int_{B_R}|v_k-(v_k)_\theta|^2\right)^{\frac 12}\!\!
+\vartheta^{\frac{\alpha}{2}}+c\left(\fuetainv\vartheta\right)^{\frac{3\alpha}{2}}\!\left(1+|(v_k)_{Q_\vartheta}-(v_k)_{Q_\frac 12}|\right)\!.
\end{align}
We take the limit in \eqref{e:contr_risc}: Lemma \ref{lem:lions} gives a pair $(v,q)$ solution of \eqref{e:linNS_1}. Taking into account \eqref{e:nl_ecc_est}, we get
\begin{align*}
\frac 12\le & \hphantom{+}E^V(v;\vartheta)+ E^P(q;\vartheta)
 +\left(\mean{-\vartheta^{2\alpha}}^0\sup_{\fueta\vartheta\le R<\fueta}\left(\frac{\vartheta}{R}\right)^{3\alpha}\frac{1}{R^{3}}\int_{B_R}|v-(v)_{Q_\vartheta}|^2\right)^\frac 12\\
& +\vartheta^{\frac{\alpha}{2}}+c\left(\fuetainv\vartheta\right)^{\frac{3\alpha}{2}}\left(1+|(v)_{Q_\vartheta}-(v)_{Q_\frac 12}|\right).
\end{align*}
By semicontinuity $E(v,q;1)\le \liminf E(v_k,q_k;1)=1$. Thus, thanks to Lemma \ref{lem:hoelder} we obtain the bounds:
$$
E^V(v;\vartheta)\le c\vartheta^\frac 13, \qquad
 E^P(p;\vartheta)\le c\vartheta^\frac 12,
\qquad |(v)_{Q_\vartheta}-(v)_{Q_\frac 12}|\le c\,.
$$
Moreover, for every $(x,t)\in Q_R$, $|v-(v)_{Q_\vartheta}|\le CR^\frac 13$, hence
\[
\left(\mean{-\vartheta^{2\alpha}}^0\sup_{\fueta\vartheta\le R<\fueta}\left(\frac{\vartheta}{R}\right)^{3\alpha}\frac{1}{R^{3}}\int_{B_R}|v-(v)_{Q_\vartheta}|^2\right)^\frac 12
\le C\left(\mean{-\vartheta^{2\alpha}}^0\sup_{\fueta\vartheta\le R<\fueta} \left(\frac{\vartheta}{R}\right)^{3\alpha} R^{\frac 23}\right)^\frac 12\le C \vartheta^{\frac 13}.
\]
Putting everything together
\[
\frac 12\le c\left(\vartheta^\frac 13 +\vartheta^\frac 12+\vartheta^{\frac 13}+\left(\fuetainv\vartheta\right)^{\frac{3\alpha}{2}}\right)\,.
\]
Since $c$ is independent of $\vartheta$, if the latter is small enough we obtain the desired contradiction.
\end{proof}

\subsection{Proof Theorem~\ref{thm:eps_reg-variant}}
\begin{proof}[Proof of Theorem~\ref{thm:eps_reg-variant}]
\noindent{\bf Step 1.} Let $\varepsilon_0$ be a constant (to be chosen later) which is smaller than the one given in Proposition \ref{prop:exc_decay} for $M=1$, allowing the decay of the excess. There exists $\varepsilon>0$ such that, if \eqref{eccessivoealternativo-variant} holds, then 
\begin{equation}\label{e:exc_start_small-variant}
E(u,p;x_0,t_0,1)\le\varepsilon_0\quad\text{ for every }(x_0,t_0)\in Q_1.
\end{equation} 

To show the claim, we observe that the quantity in \eqref{eccessivoealternativo-variant} controls $E^V(u;x_0,t_0,1)+E^P(p;x_0,t_0,1)$:
\begin{equation}\label{e:E-V-control}
E^V(u;x_0,t_0,1)\le \left(\mean{Q_1(x_0,t_0)}|u|^3\right)^\frac 13 +  C|(u)_{Q_1(x_0,t_0)}|\le C\left(\mean{Q_2(x_0,t_0)}|u|^3\right)^\frac 13\le C^V\varepsilon^\frac 13
\end{equation} 
and similarly
\begin{equation}\label{e:E-P-control}
E^P(p;x_0,t_0,1)\le C^P\varepsilon^\frac 23.
\end{equation} 

 For the nonlocal part, we notice that for every $R \geq \fueta$ we have that $B_R(x) \subseteq B_{R+1}$ and
$$\sup_{R\ge\fueta}R^{-\frac{3\alpha}{2}} \left(\mean{B_R(x_0)} |u|^{2}\right)^\frac{1}{2}
\leq \sup_{R\ge \fueta}\left( \frac{R+1}{R}\right)^{3-\frac{3\alpha}{2}} \frac{1}{(R+1)^\frac{3\alpha}{2}} \left(\mean{B_{R+1}} |u|^2\right)^\frac{1}{2}.
$$

\noindent Integrating in $[t-1,t]$ and recalling that $|(u)_{Q_1(x_0,t_0)}|\le C\|u\|_{L^3(Q_1(x_0,t_0))}$, 
\[
\left(\int_{t-1}^t\sup_{R\ge \fueta}R^{-3\alpha} \mean{B_R(x_0)} |u-(u)_{Q_1(x_0,t_0)}|^2\right)^\frac 12
\le C \left(\mean{-2^{2\alpha}}^0\sup_{R\ge \fueta}R^{-3\alpha} \mean{B_R(x_0)} |u|^2\right)^\frac 12+C|(u)_{Q_1(x_0,t_0)}|,
\]

\noindent Provided that we choose $\varepsilon$ small enough to satisfy $C^{nl}\varepsilon^\frac 13+C^V\varepsilon^\frac 13+C^P\varepsilon^\frac 23\le\varepsilon_0$, we have shown \eqref{e:exc_start_small-variant}.

\noindent{\bf Step 2.} {We claim that there exist an exponent $\gamma>0$ and a constant $C>0$ such that, if $\varepsilon_0$ in \eqref{e:exc_start_small-variant} is small enough, then}
\begin{equation}\label{e:decay_fin}
E(u,p;x_0,t_0,r)\le Cr^\gamma\quad\text{ for every }(x_0,t_0)\in Q_1, r\le 1.
\end{equation}
This proves $u\in C^{0,\alpha}\left(Q_1\right)$ thanks to Morrey's theorem (see, for instance, \cite{Campanato}).
To show the claim, we prove by induction that, provided $\varepsilon_0$ is chosen smaller than a geometric constant (which will be specified later in terms of the constant $\vartheta$ given by Proposition \ref{prop:exc_decay} with $M=1$),
\begin{equation}\label{e:induct}
E(u,p;x_0,t_0,\vartheta^k)\le \frac{E(u,p;x_0,t_0,1)}{2^k}\le \frac{\varepsilon_0}{2^k}\quad \text{ and }\quad |(u)_{Q_{\vartheta^{k}}}|\le C\varepsilon_0\sum_{i=0}^k \frac{1}{2^{i-1}}
\end{equation}
(here and in the rest of Step 2, all excesses of $(u,p)$ and all cylinders are centered at any arbitrary point $(x_0,t_0)$, and the estimates are uniform as $(x_0,t_0)$ vary in $Q_1$).
For $k=0$ the first inequality follows from Step 1 and the second comes from the fact that $|(u)_{Q_1}|\le C\|u\|_{L^3(Q_1)}\le C\varepsilon^\frac 13$, so the choice of $\varepsilon$ above already fulfills the requirement. For the inductive step from $k$ to $k+1$ we first wish to apply Proposition \ref{prop:exc_decay} with $M=1$ and $r=\vartheta^k$. The condition on the excess follows from the inductive assumption. Concerning the average, we observe
\[
|(u)_{Q_{\vartheta^{k}}}|\le 4C\varepsilon_0,
\] 
thus we simply need $\varepsilon_0\le \frac{1}{4C}$. The Proposition \ref{prop:exc_decay} gives
\[
E(u,p;x_0,t_0,\vartheta^{k+1})\le \frac 12 E(u,p;x_0,t_0,\vartheta^k)\le\frac{E(u,p;x_0,t_0,1)}{2^{k+1}}.
\]
As for the average, we observe the simple inequality
\[
|(u)_{Q_{\vartheta^{k}}}-(u)_{Q_{\vartheta^{k+1}}}|\le \frac C{\vartheta^{3+2\alpha}} \left(\mean{Q_{\vartheta^k}} |u-(u)_{Q_{\vartheta^k}}|^3\right)^\frac 13\le C E(u,p;x_0,t_0,\vartheta^k)\le C\varepsilon_0 2^{-k},
\]
so that $|(u)_{Q_{\vartheta^{k+1}}}|\le |(u)_{Q_{\vartheta^{k}}}|+C\varepsilon_0 2^{-k}\le C\varepsilon_0\sum_{i=0}^{k+1}\frac{1}{2^{i-1}}$. Finally \eqref{e:decay_fin} is implied by \eqref{e:induct}: indeed
 for $r=\vartheta^k$ \eqref{e:decay_fin} corresponds to the first claim in \eqref{e:induct} choosing $\gamma=-\log_\vartheta(2)$, and
 for $r\in (\vartheta^{k+1},\vartheta^k)$ we compare $E(u,p;x_0,t_0,r)\le C(\vartheta) E(u,p;x_0,t_0,\vartheta^k)\le C \varepsilon_0 2^{-k}\le C \theta^{\gamma(k+1)}\le C r^{\gamma}$.
\end{proof}

\section{Proof of the Caffarelli--Kohn--Nirenberg type theorem}\label{s:CKN}

To prove Theorem \ref{t:CKN}, by translation invariance of the equation we can assume, without loss of generality, that $(x_0,t_0)=(0,0)$ and therefore, for every quantity as $\fE$ we omit the dependence on the function $(u,p)$ (which we consider fixed) and the point $(x,t)$ (which we assume to be the origin).  
We introduce moreover the following \emph{scaling invariant quantities} (according to the natural rescaling \eqref{e:scaling}):
\begin{align*}
\fB(r) & :=\frac{1}{r^{3-2\alpha}}\int_{Q_r}|\nabla u (x,t)|^2\,dx\,dt\\
\fF(r) & := \frac{1}{r^{5-2\alpha}}\int_{Q_r} |u (x,t)|^2\, dx\, dt\\
\fT(r) & := r^{5\alpha -2} \int_{-r^{2\alpha}}^0 \sup_{R\geq \frac{r}{4}} \frac{1}{R^{3\alpha}}\mean{B_R} |u|^2\, dt
\end{align*}
(in particular, the \emph{tail functional} $\fT$ has been already defined in \eqref{def:tail}).
The following lemma can be regarded as a Poincar\'e-type estimate which is not really using the Navier-Stokes equation. 
\begin{lemma}\label{l:BFT_small}
Let $\alpha \in (1,3/2)$, then there exists a constant $C=C(\alpha)$ such that,
for any function $u \in 
L^2([0,T]; H^\alpha (\RR^3))$,
\[
 \limsup_{r\to 0}\; (\fB(r) + \fF(r) + \fT(r))<C \limsup_{r\to 0} \fE(r) \, .
\]
\end{lemma}
\begin{proof} {\bf Estimate on $\fB$.}
We claim that there exist $\theta\in (0,1/4)$ and $C:=C(\theta)>0$ such that 
\begin{equation}\label{e:B_da_iterare}
\fB(\theta r) \leq \frac{1}{2} \fB(r)+ C \fE(r) \qquad\qquad \forall r >0\, .
\end{equation} 
This obviously implies 
\begin{equation}\label{e:pezzo_in_B}
 \limsup_{r\to 0} \fB(r) \leq C \limsup_{r\to 0} \fE (r) \, .
\end{equation}
Moreover, without loss of generality we prove \eqref{e:B_da_iterare} for $r=1$. 

\medskip

In order to simplify our notation we set $v:= \nabla u$ and $v^\flat := (\nabla u)^\flat$. We first show the following two inequalities
\begin{align}
\int_{Q^*_1} |v^\flat|^2 &\leq C (\fB(1)+ \fE(1))\label{e:step1}\\
\fB (\theta) &\leq \frac{C}{\theta^{4-2\alpha}} \int_{-\theta^{2\alpha}}^0 \int_{B_\theta \times [\theta, 2\theta]} |v^\flat|^2+ \frac{C}{\theta^{4-2\alpha}} \fE (1)\, \label{e:step2}.
\end{align}
For every $(x,y,t)\in Q_1^*$ we estimate
\begin{align}
|v^\flat|^2 (x,y,t) &\leq 2 |v|^2 (x,t) + 2 \left(\int_0^y |\partial_y v^\flat| (x,z,t)\, dz\right)^2\nonumber\\
&\leq 2 |v|^2(x,t) + 2 \int_0^y z^{b} |\overline\nabla v^\flat|^2 (x,z,t)\, dz\int_0^y z^{2\alpha-3}\, dz\nonumber\\
&\leq 2 |v|^2 (x,t) + \frac{2}{2\alpha-3} \int_0^y z^{b} |\overline\nabla v^\flat|^2 (x,z,t)\, dz\,.\label{e:diretta}
\end{align}
Similarly
\begin{equation}\label{e:contraria}
 |v|^2 (x,t) \leq 2 |v^\flat|^2 (x,y,t) + \frac{2}{2\alpha-3} \int_0^y z^{b} |\nabla v^\flat|^2 (x,z,t)\, dz\,.
\end{equation}
Integrating \eqref{e:diretta} on $Q_1^*$ we easily conclude \eqref{e:step1}.
Integrating \eqref{e:contraria} in $(x,y,t)$ on $B_\theta\times [\theta, 2\theta]\times (-\theta^{2\alpha},0]$ and dividing by $\theta^{4-2\alpha}$ we reach \eqref{e:step2}.
Next we compute
\begin{align}
\int_{-\theta^{2\alpha}}^0 \int_{B_\theta \times [\theta, 2\theta]} |v^\flat|^2  &\leq C \int_{-\theta^{2\alpha}}^0 \int_{B_\theta\times [\theta, 2\theta]} \left|v^\flat (x,y,t) - \mean{B_1\times [\theta,1]} v^\flat (\cdot , t)\right|^2
+ C \theta^4 \int_{-\theta^{2\alpha}}^0 \left|\mean{B_1\times [\theta,1]} v^\flat (\cdot , t)\right|^2\nonumber\\
&\leq C  \int_{-\theta^{2\alpha}}^0 \int_{B_1\times [\theta,1]} \left|v^\flat -   \mean{B_1\times [\theta, 1]} v^\flat\right|^2 
+ C \theta^4 \int_{Q_1^*} |v^\flat|^2\nonumber\\
&\leq C \int_{-\theta^{2\alpha}}^0 \int_{B_1\times [\theta,1]} |\overline\nabla v^\flat|^2 + C \theta^4 \int_{Q_1^*} |v^\flat|^2\nonumber\\
&\leq \frac{C}{\theta^{3-2\alpha}} \int_{Q_1^*} y^{b} |\overline\nabla v^\flat|^2 +C \theta^4 \int_{Q_1^*} |v^\flat|^2 
= C\theta^4 \int_{Q_1^*} |v^\flat|^2 + \frac{C}{\theta^{3-2\alpha}} \fE(1)\, .\label{e:step3}
\end{align}
Using \eqref{e:step2}, \eqref{e:step3} and  \eqref{e:step1}, we have that
\[
\fB (\theta) \leq C \theta^{2\alpha}\int_{Q_1^*}|v^\flat|^2 + C \left(\theta^{2\alpha-4} +\theta^{4\alpha-7}\right) \fE (1)\leq C \theta^{2\alpha}\fB(1) + C \left(\theta^{2\alpha}+\theta^{2\alpha-4} +\theta^{4\alpha-7}\right) \fE (1)\, .
\]
Choosing $\theta$ appropriately small, we conclude \eqref{e:B_da_iterare}.

\bigskip

{\bf Estimate on $\fF$.} We proceed similarly as above and claim that there exists $\theta\in (0,1/4)$ such that
\begin{equation}\label{e:F_da_iterare}
\fF(\theta r) \leq \frac{1}{2} \fF (r) + C \fB(r)\, .
\end{equation}
Again we prove the claim, without loss of generality, for $r=1$. Indeed
\begin{align*}
\fF (\theta) &\leq \frac{2}{\theta^{5-2\alpha}} \int_{Q_\theta} \left|u - \mean{B_1} u\right|^2 +  2\theta^{2\alpha-2}  \int_{-\theta^{2\alpha}}^0 \left|\mean{B_1} u\right|^2\leq \frac{2}{\theta^{5-2\alpha}} \int_{Q_1} \left|u - \mean{B_1} u\right|^2 + 2\theta^{2\alpha-2} \int_{Q_1} |u|^2\\
&\leq \frac{C}{\theta^{5-2\alpha}} \int_{Q_1} |\nabla u|^2 + C \theta^{2\alpha-2} \fF(1) = \frac{C}{\theta^{5-2\alpha}} \fB(1) + C \theta^{2\alpha-2} \fF(1)\, ,
\end{align*}
and choosing $\theta$ appropriately we reach the desired conclusion.

\bigskip

{\bf Estimate on $\fT$.} As above we claim that there exists $\theta\in (0,1/4)$ such that
\begin{equation}\label{e:T_da_iterare}
\fT (\theta r) \leq \frac{1}{2} \fT (r) + C \fB(r)
\end{equation}
and we assume, without loss of generality, that for $r=1$.

Let $S(t)$ be the function $S(t) := \sup_{R\geq \frac{1}{4}} \frac{1}{R^{3\alpha}} \textmean{B_R} |u|^2$. Let $\rho \in [\frac{\theta}{4}, \frac{1}{4})$. We simply compute
\begin{align*}
\mean{B_\rho} |u|^2 &\leq 2 \mean{B_\rho} \left|u-[u]_{B_1}\right|^2 + 2 \left|\mean{B_1} u\right|^{2}
\leq \frac{C}{\rho^3} \int_{B_1}  \left|u-[u]_{B_1}\right|^{2} + C S(t)
\leq \frac{C}{\theta^3} \int_{B_1} |Du|^2 + C S(t)\,.
\end{align*}
In particular 
\[
\sup_{\frac{\theta}{4} \leq R} \frac{1}{R^{3\alpha}} \mean{B_R} |u|^2
\leq \frac{C}{\theta^{3\alpha}} S(t) + \frac{C}{\theta^{3 +3\alpha}} \int_{B_1} |Du|^2\, .
\]
Integrating in time and multiplying by $\theta^{5\alpha-2}$ we obtain
\[
\fT (\theta) \leq C \theta^{2\alpha-2} \fT (1) + \frac{C}{\theta^{5-2\alpha}} \fB(1)\, ,
\]
from which the desired conclusion follows choosing $\theta$ sufficiently small. 
\end{proof}

In order to prove Theorem \ref{t:CKN} we introduce the following further quantities:
\begin{align}
\fA(r) &:= \sup_{-r^{2\alpha} \leq t\leq 0} \frac{1}{r} \int_{B_r} |u|^2 (x,t)\, dx\\
\fC(r) &:= \frac{1}{r^{6-4\alpha}} \int_{Q_r} |u|^3\,dx\,dt\\
\fD(r) &:= \frac{1}{r^{6-4\alpha}} \int_{Q_r} |p|^{\frac{3}{2}}\,dx\,dt\, .
\end{align}
We will need the following interpolation inequality.

\begin{lemma}\label{lemma:CKN-supporto}
Let $u: L^{\infty}((0,T), L^2(\RR^3;\RR^3)) \cap L^2((0,T), H^{\alpha}(\RR^3;\RR^3))$. Then the following inequality holds for every $\rho, r \in \RR^+$ with $\rho \le r$:
\begin{equation}
\label{eqn:C-est}
\fC (\rho) \leq C\Big(\frac{\rho}{r}\Big)^{6\alpha-3} \fA(r)^\frac 32+ C \Big(\frac{r}{\rho}\Big)^{6-4\alpha}\fA(r)^\frac 34 \fB(r)^\frac 34 
\leq C \Big(\frac{\rho}{r}\Big)^{6\alpha-3} \fA(r)^{\frac{3}{2}} + C \Big(\frac{r}{\rho}\Big)^{9-2\alpha} \fB(r)^{\frac{3}{2}}\, .
\end{equation}
\end{lemma}
\begin{proof}
Since all the quantities are scaling-invariant, it is enough to prove the inequality for $r=1$. By the triangular inequality, we split $\fC(\rho)$ in two terms, namely
\[
\fC(\rho)=\frac{C}{\rho^{6-4\alpha}}\int_{Q_{\rho}}|u|^3 \, dx \,dt \le C\rho^{4\alpha-3}\int_{-\rho^{2\alpha}}^0\left|[u]_1(\tau)\right|^3\,d\tau+\frac{C}{\rho^{6-4\alpha}}\int_{Q_{\rho}}|u(x,\tau)-[u]_1(\tau)|^3\,dx\,d\tau.
\]
We estimate the first summand in the right hand side
\begin{align*}
\rho^{4\alpha-3}\int_{-\rho^{2\alpha}}^0\left|\mean{B_1}u(x,\tau)\,dx\right|^3\,d\tau & \le C\rho^{6\alpha-3}\sup_{-\rho^{2\alpha}\le t\le 0}\left(\int_{B_1}|u|^2(x,t)\,dx\right)^\frac 32 
\le C\rho^{6\alpha-3} \fA(1)^\frac 32.
\end{align*}
For the second summand, we firstly notice that, by H\"older's and Sobolev's inequalities,
\begin{align}\label{eqn:annalisa}
\int_{Q_{\rho}}|u(x,\tau)-[u]_1(\tau)|^3\,dx & \le \int_{Q_1}|u(x,\tau)-[u]_1(\tau)|^3\,dx\\
& \le \int_{-1}^0 \left(\int_{B_1}|u(x,\tau)-[u]_1(\tau)|^2\,dx\right)^\frac 34\left(\int_{B_1}|u(x,\tau)-[u]_1(\tau)|^6\,dx\right)^\frac 14 \, dt \\
& \le C \fA(1)^\frac 34\int_{-1}^0 \left(\int_{B_1}|Du|^2(x,\tau)\,dx\right)^\frac 34\, dt
\le C \fA(1)^\frac 34 \fB(1)^\frac 34
.
\end{align}
Putting together the last three estimates, we reach the first inequality in \eqref{eqn:C-est}; the second inequality follows by Young inequality.
\end{proof}

\begin{proof}[Proof of Theorem \ref{t:CKN}] The key point is that there are constants $\theta>0$, $C>0$ and $\omega_0$ such that,
if
\begin{equation}\label{e:condizione_omega}
\fB(r) + \fF(r) + \fT(r) = \omega < \omega_0\, ,
\end{equation}
then
\begin{equation}\label{e:tesi_omega}
\fA (\theta r)^{\frac{3}{2}} + \fD (\theta r)^2 \leq \frac{1}{2} (\fA(r)^{\frac{3}{2}} + \fD(r)^2) + C \omega^{\frac{3}{2}}\, .
\end{equation}
Once this claim is proved, we combine it with the estimate of Lemma~\ref{lemma:CKN-supporto} with $\rho=r$, namely
\[
\fC(r) \leq C \fA(r)^{\frac{3}{4}} \fB(r)^{\frac{3}{4}} + C \fA(r)^{\frac{3}{2}}\, .
\]
We can then apply Lemma \ref{l:BFT_small} and
we easily conclude that, if $\varepsilon$ is as in Theorem \ref{thm:eps_reg-variant}, for a suitably small $\delta$, the condition
\[
\limsup_{r\downarrow 0} \fE (r) < \delta
\]
implies
\[
\limsup_{r\downarrow 0} \left(\fC(r) + \fD(r) + \fT(r)^{\frac{1}{2}}\right) < \varepsilon\, .
\]
In particular, if $2r$ is a radius at which 
\[
\fC(2 r) + \fD(2 r) + \fT(2 r)^{\frac{1}{2}} < \varepsilon\, ,
\]
with an obvious scaling argument, we can apply Theorem \ref{thm:eps_reg-variant} to $u_r$ and $p_r$, hence concluding that the latter are H\"older continuous in $Q_1$. This, however, implies that $u$ and $p$ are H\"older continuous in $Q_r$, i.e., $(0,0)$ is a regular point. We are thus left with proving \eqref{e:tesi_omega}. 

\medskip

{\it Step 1: energy inequality.}
Without loss of generality we show \eqref{e:tesi_omega} with $r=1$ assuming \eqref{e:condizione_omega} with $r=1$. From the energy inequality in Lemma~\ref{lemma:suitable} applied with $M=0$ and $f(t) = |[u]_1|^2$ we conclude
\begin{equation}\label{e:90}
\fA (\theta) \leq \frac{C}{\theta^2} \int_{Q_{2\theta}} \left(|u-[u]_1||u+[u]_1| +|p||u|\right) + c (\theta) \int_{Q_{2\theta}} |u|^2 
+ c (\theta) \int_{Q_{2\theta}^*} y^b |u^*|^2\, .
\end{equation}
For the term involving the pressure, we use the inequality
\[
\frac{1}{\theta^2}\int_{Q_{2\theta}}|p||u|
\leq \frac{C}{\theta^2}
\left(\int_{Q_{2\theta}} |p| ^{\frac{3}{2}} \right)^{\frac{2}{3}}\left(\int_{Q_{2\theta}}|u|^{3} \right)^{\frac{1}{3}} = C \theta^{4\alpha-4}\fC (2\theta)^{\frac{1}{3}}\fD(2\theta)^{\frac{2}{3}}
 \leq C \theta^{8-8\alpha} \fC (2\theta)^{\frac{2}{3}} + C \fD(2\theta)^{\frac{4}{3}}\, .
\]
Using this inequality and Lemma~\ref{lemma:suitable} and \eqref{e:condizione_omega} to control the last two terms in \eqref{e:90}, we conclude
\begin{equation}\label{en-ineq-applied}
\fA (\theta) \leq \frac{C}{\theta^2} \int_{Q_{2\theta}} |u-[u]_1||u+[u]_1| |u| + C \theta^{8-8\alpha} \fC (2\theta)^{\frac{2}{3}} + C \fD(2\theta)^{\frac{4}{3}}+ c (\theta) \omega\, .
\end{equation}

\medskip
{\it Step 2: estimate of the nonlinear term in \eqref{en-ineq-applied}.}
By H\"older and Sobolev inequality we have 
\begin{align*}
&\int_{Q_{2\theta}} |u-[u]_1| |u + [u]_1| |u|\, dx \, dt \le  \left( \int_{Q_{2\theta}} |u|^{3}\, dx \, dt \right)^{\frac 13}  \left(
\int_{Q_1} |u-[u]_1|^{\frac 32} |u + [u]_1|^{\frac 32}\, dx \, dt\right)^{\frac 23}
\\&\le {C\theta^{2-\frac2 3 \alpha}} \fC(2\theta)^{\frac 13}  \left(
\int_{-1}^0 \left(\int_{B_1} |u-[u]_1|^{6}\, dx \right)^{\frac 14} \left( \int_{B_1} |u + [u]_1|^{2}\, dx \right)^{\frac 34} \, dt\right)^{\frac 23}
\\&\le {C\theta^{2-\frac2 3 \alpha}} \fC(2\theta)^{\frac 13}  \left(
\int_{-1}^0 \left(\int_{B_1} |Du|^{2}\, dx \right)^{\frac 34} \left( \int_{B_1} |u|^{2}\, dx \right)^{\frac 34} \, dt\right)^{\frac 23}
\\&\le {C\theta^{2-\frac2 3 \alpha}} \fC(2\theta)^{\frac 13}  \left(
\int_{-1}^0 \left(\int_{B_1} |Du|^{2}\, dx \right)^{\frac 34} \fA(1)^{\frac 34} \, dt\right)^{\frac 23}
\\&\le {C\theta^{2-\frac2 3 \alpha}} \fC(2\theta)^{\frac 13} \fA(1)^{\frac{1}{2}} \fB(1)^{\frac{1}{2}}
 \, .
\end{align*}

{\it Step 3: estimate on the pressure in \eqref{en-ineq-applied}.} We claim the following estimate for $\fD(2\theta)$:
\begin{equation}\label{e:stima_D2}
\fD (2\theta) \leq C \frac{1}{\theta^{6-4\alpha}} \fA(1)^{\frac{3}{4}} \fB(1)^{\frac{3}{4}} + C \theta^{4\alpha-3} \fD(1)\, .
\end{equation}
The proof uses the usual elliptic equation for the pressure.
Indeed, recalling that $u$ is divergence-free, we see that the pressure $p$ solves
\begin{align*}
\Delta p={\rm div}\,{\rm div}\,(u\otimes u)=\sum_{i,j}\partial_i\partial_j(u_iu_j)=\sum_{i,j}\partial_i\partial_j\left((u_i-[u_i]_1)(u_j-[u_j]_1)\right).
\end{align*}
Let $\chi_{B_1}$ being the characteristic function of the ball $B_1$
and consider the potential theoretic solution $\overline p$ of
\[
\Delta\overline p=\sum_{i,j}\partial_i\partial_j\left((u_i-[u_i]_1)(u_j-[u_j]_1)\chi_{B_1}\right)\, .
\]
The difference $\overline p-p$ is harmonic in $B_1$. Therefore
\begin{align*}
\int_{B_{2\theta}}|\overline p(x,\tau)-p(x,\tau)|^\frac 32dx&\le \theta^3\|\overline p(x,\tau)-p(x,\tau)\|_{L^\infty(B_{1/2})}^\frac 32
\le c \theta^{3}\int_{B_1}|\overline p(x,\tau)-p(x,\tau)|^\frac 32dx\\
&\le c {\theta}^3\int_{B_1}|p|^\frac 32(x,\tau)\,dx+c{\theta}^3\int_{B_1}|\overline p|^\frac 32(x,\tau)\,dx.
\end{align*}
Hence, by the previous inequality and the Calder\'on-Zygmund estimates on $\overline p$, we conclude
\begin{align*}
\frac{1}{\theta^{6-4\alpha}}\int_{Q_{2\theta}}|p|^\frac 32(x,\tau)\,dx\,d\tau
&\le c \theta^{4\alpha-3} \fD(1)+\frac{c}{\theta^{6-4\alpha}}\int_{Q_1}|\overline p|^\frac 32(x,\tau)\,dx\,d\tau\\
&\le c \theta^{4\alpha-3}\fD(1)+\frac{c}{\theta^{6-4\alpha}}\int_{Q_1}\left|u(x,\tau)-[u]_1\right|^3\,dx\,d\tau.
\end{align*}
Using again \eqref{eqn:annalisa}, \eqref{e:stima_D2} follows.

{\it Step 4: conclusion.}
Taking the power $3/2$ of \eqref{en-ineq-applied} and noticing that $\fD(\theta) \leq C \fD(2\theta)$, we obtain
\begin{equation}\label{en-ineq-applied-new}
\fA (\theta)^{\frac 32} +\fD(\theta)^2 \leq \frac{C}{\theta^3}\Big( \int_{Q_{2\theta}} |u-[u]_1||u+[u]_1| |u| \Big)^{\frac 32}+ C \theta^{12-12\alpha} \fC (2\theta) + C \fD(2\theta)^{2}+ c (\theta) \omega^{\frac 32}\, .
\end{equation}
{Applying Step 2 and Step 3
\begin{align}
\fA (\theta)^{\frac 32} +\fD(\theta)^2 & \leq \frac{C}{\theta^{\alpha}} \fC(2\theta)^{\frac 12} \fA(1)^{\frac{1}{2}} \fB(1)^{\frac{1}{2}}+ C \theta^{12-12\alpha} \fC (2\theta) + C \fD(2\theta)^{2}+ c (\theta) \omega^{\frac 32}\nonumber\\
&\leq \frac{C}{\theta^{\alpha}} \fC(2\theta)^{\frac 12} \fA(1)^{\frac{3}{4}} \fB(1)^{\frac{3}{4}} + C \theta^{12-12\alpha} \fC (2\theta) + \frac{C}{\theta^{6-4\alpha}} \fA(1)^{\frac{3}{4}} \fB(1)^{\frac{3}{4}} + 
C \theta^{4\alpha -3} \fD (1) + c (\theta) \omega^{\frac 32}\, .\nonumber
\end{align}
We can easily bound
\[
\frac{C}{\theta^{6-4\alpha}}  \fA(1)^{\frac{3}{4}} \fB(1)^{\frac{3}{4}} \leq \frac{1}{8} \fA (1)^{\frac{3}{2}} + c(\theta) \omega^{\frac{3}{2}} 
\]
and
\[
\frac{C}{\theta^{\alpha}} \fC(2\theta)^{\frac 12} \fA(1)^{\frac{1}{2}} \fB(1)^{\frac{1}{2}}\leq
C \theta^{12-12\alpha} \fC (2\theta) + \frac{1}{8} \fA (1)^{\frac{3}{2}}\, ,
\]
(since we can assume without loss of generality that $\fB (1) \leq \omega_0 \leq 1$).
We therefore achieve
\[
\fA (\theta)^{\frac 32} +\fD(\theta)^2 \leq C \theta^{12-12\alpha} \fC (2\theta) + \frac{1}{4} \fA (1)^{\frac{3}{2}}  + c (\theta) \omega^{\frac{3}{2}}\, .
\]
We now use Lemma \ref{lemma:CKN-supporto} to derive
\[
\fC (2\theta) \leq C \theta^{6\alpha -3} \fA(1)^{\frac{3}{2}} + C \theta^{4\alpha -3} \fD (1)^2 + c (\theta) \omega^{\frac{3}{2}}\, .
\]
In particular we conclude
\[
\fA (\theta)^{\frac 32} +\fD(\theta)^2 \leq \left(\frac{1}{4} + C \theta^{9-6\alpha}\right) \fA (1) + C \theta^{4\alpha -3} \fD (1)^2 + c (\theta) \omega^{\frac{3}{2}}\, .
\]
Since $9-6\alpha$ and $4\alpha -3$ are both positive,
we just need to choose $\theta$ sufficiently small to conclude \eqref{e:tesi_omega}. }
\end{proof}

\section{Proofs of Theorems \ref{thm-intro:sing-set-dim}, \ref{thm:eps_reg-massimale} and Corollaries \ref{c:box},\ref{c:stab}, \ref{c:reg}}\label{s:final}
\subsection{Dimension of the singular set}

\begin{definition}
Given a parabolic cylinder 
\[
Q_r(x,t)=B_r(x)\times]t_0-r^{2\alpha},t_0]=B_r(x)\times \left]\left(t-\frac{r^{2\alpha}}{2}\right)-\frac{r^{2\alpha}}{2},\left(t-\frac{r^{2\alpha}}{2}\right)+\frac{r^{2\alpha}}{2}\right]
\]
and $\lambda>0$, we define its dilation (computed with respect to its centroid)
\[
\lambda Q_r(x,t):=B_{\lambda r}(x)\times \left]\left(t-\frac{r^{2\alpha}}{2}\right)-\frac{(\lambda r)^{2\alpha}}{2},\left(t-\frac{r^{2\alpha}}{2}\right)+\frac{(\lambda r)^{2\alpha}}{2}\right]\,.
\]
\end{definition}
Notice that these parabolic cylinders have the ``Vitali property'' with $\lambda=5$, namely
\[
Q_{r}(x_1,t_1)\cap Q_{2r}(x_2,t_2)\neq\emptyset\quad\Longrightarrow\quad Q_{2r}(x_2,t_2)\subset  5Q_r(x_1,t_1)\,.
\]
In particular we recover the classical $5r$-covering lemma. 

\begin{lemma}[Vitali's covering theorem for cylinders]\label{thm:vitali}
Let ${\mathcal F}\subset ]0,+\infty[\times\RR^3\times\RR$ be a family of parameters for the closed parabolic cylinders in ${\mathcal Q}=\{\overline{Q_r(x,t)}:\,(r,x,t)\in {\mathcal F}\}$ and assume $\sup_{\mathcal F}r<\infty$.
Then there exists a countable family ${\mathcal G}\subset{\mathcal F}$ which consists of disjoint cylinders and such that
\[
\bigcup_{(r,x,t)\in\mathcal F} \overline{Q_r(x,t)}\subset\bigcup_{(r,x,t)\in{\mathcal G}}\overline {5 Q_r(x,t)}\,.
\]
\end{lemma}

\begin{proof}[Proof of Theorem~\ref{thm-intro:sing-set-dim}] Let $\eps>0$ be the threshold given by Theorem~\ref{t:CKN}.
Fix $\delta>0$ and let
\[
{\mathcal Q}_\delta:=\left\{\overline{Q_r(x,t)}:\,r<\delta\text{ and }\int_{Q^*_r(x,t)} y^{b}|\nabla (\nabla u)^\flat |^2\ge\varepsilon r^{5-4\alpha}\right\}\,.
\]
Observe that, by Theorem \ref{t:CKN}, the set of points $(x,t)\in\RR^3\times\RR$ such that $\overline{Q_r(x,t)}\in {\mathcal Q}_\delta$ for some $r$ contains ${\rm Sing}\, u$. Then, by Lemma \ref{thm:vitali}, there exists a countable family $\{(r_k,x_k,t_k)\}_{k\ge 1} \subset {\mathcal Q}_\delta$ such that 
\[
{\rm Sing}\, u\subset \bigcup_{k\ge 1} \overline{5Q_{r_k}(x_k,t_k)}\,.
\]
Thus, since the $Q_{r_k}^*(x_k,t_k)$ are pairwise disjoint, if we set $R_\delta:=\bigcup_{k\ge 1} Q^*_{r_k}(x_k,t_k)$, then we can estimate
\begin{equation}\label{e:squeeze_sing}
\PHaus{5-4\alpha}_\delta\left({\rm Sing}\, u\right)\le \sum_{k\ge 1}r^{5-4\alpha}_k\le\sum_{k\ge 1}\frac{1}{\varepsilon}\int_{Q^*_{r_k}(x_k,t_k)}y^{b}|\nabla(\nabla u)^\flat|^2\le\frac{1}{\varepsilon}\int_{R_\delta}y^{b}|\nabla(\nabla u)^\flat|^2\,.
\end{equation}
Moreover $R_\delta\subset \RR^3\times[0,\delta]\times[0,+\infty)$; therefore, by absolute continuity of the integral (with respect to the finite measure $y^{b}|\nabla(\nabla u)^\flat|^2\,dx\,dy\,dt$),
\[
\lim_{\delta\to 0}\int_{R_\delta}y^{b}|\nabla(\nabla u)^\flat|^2=0.
\]
Thus we conclude \eqref{est:dim_singset} from \eqref{e:squeeze_sing}.
\end{proof}

\begin{proof}[Proof of Corollary \ref{c:box}] Let $\alpha \in  (1, \frac{5}{4})$ and let $u$ be a suitable weak solution of \eqref{e:NS_alfa}. Fix $t>0$ and set $S:= {\rm Sing}\, u\cap (\RR^3\times [t, \infty))$. We recall the definition of the box-counting dimension ${\rm Dim}_b\, (S)$: for every fixed $\delta\in (0,1)$ we consider the minimal number $N (\delta)$ of sets of diameter $\delta$ which are needed to cover $S$ and we set
\[
{\rm Dim}_b\, (S) = \limsup_{\delta\downarrow 0}  (- \log_\delta (N (\delta)))\, .
\]
Fix therefore $\delta \in (0, \min (t^{\sfrac{1}{2\alpha}}, 1))$. We will indeed estimate the minimal number $N' (\delta)$ of $Q_\delta (x_i, t_i)$ which are needed to cover $S$, because it is easy to see that
\[
{\rm Dim}_b\, (S) \leq \limsup_{\delta\downarrow 0} (- \log_\delta (N' (\delta)))\, .
\]
Recall that $u\in L^\infty ((0,\infty), L^2(\RR^3)) \cap L^2 ((0,\infty), H^\alpha(\RR^3))$. In particular, by Sobolev's embedding 
$u\in L^2 ((0, \infty), L^{\frac{6}{3-2\alpha}}(\RR^3))$ and, by interpolation, $u\in L^{\frac{6+4\alpha}{3}} (\RR^3\times (0, \infty))$. Using the Calder\'on--Zygmund estimates and the usual maximal function estimates we then get
\[
M := \int_{\RR^3\times (0, \infty)} \left(\mathcal{M} |u|^2 + p \right)^{\frac{3+2\alpha}{3}} < \infty\, .
\]
Next, by H\"older's inequality, 
\[
\frac{1}{r^{6-4\alpha}} \int_{Q_r (y,s)} \left(\mathcal{M} |u|^2 + p\right)^{\frac{3}{2}}
\leq \left(\frac{1}{r^{\frac{15-2\alpha -8\alpha^2}{3}}} \int _{Q_r (y,s)} \left(\mathcal{M} |u|^2 + p\right)^{\frac{3+2\alpha}{3}}\right)^{\frac{9}{6+4\alpha}}\, .
\]
In particular, by Theorem \ref{thm:eps_reg-massimale}, if $(y,s)\in S$, then
\[
\int_{Q_r (y,s)} \left(\mathcal{M} |u|^2 + p\right)^{\frac{3+2\alpha}{3}} \geq \varepsilon^{\frac{6+4\alpha}{9}} r^{\frac{15-2\alpha -8\alpha^2}{3}} \qquad \forall r\in (0, s^{\frac{1}{2\alpha}})\, .
\]
Assume, therefore, that $\{Q_{\delta/5} (x_i, t_i)\}_i$ is a (at most) countable cover of $S$ with $(x_i, t_i)\in S$. Using Lemma \ref{thm:vitali}, there is a subset $A\subset \NN$ such that $\{5 Q_{\delta/5} (x_i, t_i)\}_{i\in A}$ covers $S$ and 
$\{Q_{\delta/5} (x_i, t_i)\}_{i\in A}$ consists of disjoint cylinders. We thus conclude that
\[
N' (\delta) \leq \sharp A \leq \frac{M}{\varepsilon^{\frac{6+4\alpha}{9}} r^{\frac{15-2\alpha -8\alpha^2}{3}}}\, ,
\]
which in turn implies
\[
- \log_\delta (N' (\delta)) \leq - \log_\delta \left(\frac{M}{\varepsilon^{\frac{6+4\alpha}{9}}}\right) + \frac{15-2\alpha -8\alpha^2}{3}\, .
\]
In particular, letting $\delta\downarrow 0$ we reach
\[
{\rm Dim}_b\, (S) \leq \limsup_{\delta\downarrow 0} (- \log_\delta (N' (\delta))) \leq \frac{15-2\alpha -8\alpha^2}{3}\, .\qedhere
\]
\end{proof}

\subsection{$\eps$-regularity with the maximal function}
\begin{proof}[Proof of Theorem~\ref{thm:eps_reg-massimale}]

We claim that for every  $\varepsilon_0$ there exists $\varepsilon>0$ such that, if \eqref{eccessivoealternativo} holds, then 
\begin{equation}\label{e:exc_start_small}
E(u,p;x_0,t_0,1)\le\varepsilon_0\quad\text{ for every }(x_0,t_0)\in Q_1.
\end{equation} 
Once this claim is proved, we use the argument of Step 2 in the proof of Theorem~\ref{thm:eps_reg-variant} and we choose $\varepsilon_0$ sufficiently small to deduce the decay of the excess \eqref{e:decay_fin} and that $u\in C^{0,\alpha}\left(Q_1\right)$ from Morrey's theorem.

To show \eqref{e:exc_start_small}, we observe that two o the three terms in the definition of the excess have been estimated in \eqref{e:E-V-control} and \eqref{e:E-P-control}:  thanks to \eqref{eccessivoealternativo} we achieve
\[
E^V(u;x_0,t_0,1) + E^P(p;x_0,t_0,1)\le C\left(\mean{Q_2(x_0,t_0)}|u|^3\right)^\frac 13 + \le C\left(\mean{Q_2(x_0,t_0)}|p|^{3/2}\right)^\frac 23 \le C^V\varepsilon^\frac 13 + C^P\varepsilon^\frac 23.
\]
Concerning the nonlocal part, we have
\[
\left(\int_{t_0-1}^{t_0}\sup_{R\ge\fueta}R^{-3\alpha} \mean{B_R(x_0)} |u-(u)_{Q_1(x_0,t_0)}|^2\right)^\frac 12
\le \left(\int_{t_0-1}^{t_0}\sup_{R\ge\fueta}R^{-3\alpha} \mean{B_R(x_0)} |u|^{2}\right)^\frac 12+C|(u)_{Q_1(x_0,t_0)}|\, .
\]
The second summand in the right hand side is estimated by $|(u)_{Q_1(x_0,t_0)}|\le C\|u\|_{L^3(Q_1(x_0,t_0))}$.
Concerning the first summand, for every $R\ge\fueta$ and $t\in [t_0-1,t_0]$, we have that
\begin{equation*}
\mean{B_R(x_0)} |u|^2 \le C\int_{B_\fueta(x_0)}\mean{B_{R+\fueta}(x)} |u|^2\,dx'\,dx\le C\int_{B_\fueta(x_0)}{\mathcal M}|u|^2(x) \, dx.
\end{equation*}
Hence the nonlocal excess can be controlled by suitable norm of the maximal function of $|u|^2$ and by  \eqref{eccessivoealternativo}
\[
E^{nl}(u;x_0,t_0,1)\le C\left(\int_{t_0-1}^{t_0} \int_{B_\fueta(0)}{\mathcal M}|u|^2\right)^\frac 12+ C\left(\int_{Q_1}|u|^3\right)^\frac 13\\
\le C\left(\int_{Q_2} \left({\mathcal M}|u|^2\right)^{\frac{3}{2}}\right)^\frac 13\le C^{nl}\varepsilon^\frac 13.
\]
Provided we choose $\varepsilon$ small enough to satisfy $C^{nl}\varepsilon^\frac 13+C^V\varepsilon^\frac 13+C^P\varepsilon^\frac 23\le\varepsilon_0$, we obtain \eqref{e:exc_start_small}.
\end{proof}

\subsection{Stability of the regular set} 

\begin{proof}[Proof of Corollary \ref{c:stab}] Following the proof of Lemma \ref{lem:lions} we conclude easily that $\{u_k\}$ is strongly precompact in $L^3 (\mathbb R^3 \times [0,T])$ and $\{p_k\}$ is strongly precompact in $L^{3/2} (\mathbb R^3\times [0,T])$. In particular, by classical estimates on the 
maximal function, $\mathcal{M} |u_k|^2$ converges strongly in $L^{3/2}$ to $\mathcal{M} |u|^2$. On the other hand, $\mathcal{M} |u|^2\in L^\infty (Q_{2r} (x_0, t_0))$. In particular, there is a $\rho>0$ sufficiently small such that, for any $(x,t)\in Q_r (x_0, t_0)$,
\[
\frac{1}{\rho^{6-4\alpha}} \int_{Q_\rho (x,t)} \left(\mathcal{M} |u|^2 + |p|\right)^{\sfrac{3}{2}} \leq \frac{\varepsilon}{2}\, ,
\] 
where $\varepsilon$ is the constant of Theorem \ref{thm:eps_reg-massimale}. Thus, for $k$ large enough, we have
\[
\frac{1}{\rho^{6-4\alpha}} \int_{Q_\rho (x,t)} \left(\mathcal{M} |u_k|^2 + |p_k|\right)^{\sfrac{3}{2}} \leq \varepsilon \qquad\mbox{for every $(x,t)\in Q_\rho (x_0, t_0)$.}
\]
In particular Theorem \ref{thm:eps_reg-massimale} implies that for every such $k$ every point $(x,t)\in Q_\rho (x_0, t_0)$ is regular.
\end{proof}

\begin{proof}[Proof of Corollary~\ref{c:reg}] We first show the regularity of the solution. We argue by contradiction and assume therefore that the statement is false: we conclude that there is a sequence of initial data $\{u_0^k\}_k\subset Y$, a sequence $\alpha_k\uparrow \frac{5}{4}$ and a corresponding sequence of suitable weak solutions $(u_k, p_k)$ of \eqref{e:NS_alfa} with $\alpha = \alpha_k$ and $u_k  (\cdot, 0) = u^k_0$ such that none of the $u_k$ is regular. First of all, arguing as above, we can assume that, up to subsequences $(u_k, p_k)$ converge to a solution $(u,p)$ of \eqref{e:NS_alfa} with $\alpha = \frac{5}{4}$ and initial data $u (\cdot, 0)\in \overline{Y}$. Moreover we conclude the strong convergence of $\{u_k\}$ and $\{p_k\}$ respectively  in $L^3 (\mathbb R^3 \times [0,T])$ and $L^{\sfrac{3}{2}} (\mathbb R^3 \times [0,T])$ for every $T$. Note in particular that, by standard arguments, if
$|x|+|t|$ is larger than some fixed constant $M$, 
\[
\int_{Q_2 (x,t)} \left(\mathcal{M} |u_k|^2 + |p_k|\right)^{\sfrac{3}{2}} \leq \varepsilon\, .
\]
Thus, any singular point for any solution $(u_k, p_k)$ is contained in $B_M \times [0, M]$ and $u_k$ is uniformly bounded in $L^\infty ((\RR^3\times (0, \infty))\setminus (B_M  \times [0,M])$. On the other hand, we also know that, for a universal $T_0>0$, the solutions $(u_k, p_k)$ are classical, and hence regular, on $\mathbb R^3\times [0, T_0]$. Thus, if $(x_k, t_k)$ is a singular point for $(u_k, p_k)$ we can assume, up to subsequences, that $(x_k, t_k)$ converges to $(x,t)\in \overline{B}_M  \times [T_0, M]$. Since $(u,p)$ is regular, 
Corollary \ref{c:stab} gives a contradiction. 

Observe that, as a corollary of the above argument, we can bound the $L^\infty$-norm of $u_k$ for $k$ sufficiently large. Thus, the weak-strong uniqueness statement of Theorem \ref{t:Leray} implies that, for $k$ sufficiently large, $u_k$ is also the unique Leray--Hopf weak solution with initial data $u^k_0$. 
\end{proof}

\appendix

\section{Existence and weak-strong uniqueness}

\subsection{Existence of Leray--Hopf weak solutions}\label{ss:Leray_ex}
In this section we prove the existence part of Theorem \ref{t:Leray}. Consider a family of standard mollifiers $\phi_\varepsilon$ in space and define the following approximation of the hyperdissipative system:
\begin{equation}\label{e:NS_alfa_moll}
\left\{
\begin{array}{l}
\partial_t u + (u\ast \phi_\varepsilon \cdot \nabla) u + \nabla p = - (-\Delta)^{\alpha} u\\ \\
\div u =0
\end{array}\right.
\end{equation}
with the initial condition
$u (\cdot, 0) = u_0\ast \phi_\varepsilon$.
Very standard arguments show the local in time existence of a smooth solution of \eqref{e:NS_alfa_moll}, cf. for instance \cite{TangYu}. Denote by $(u_\varepsilon, p_\varepsilon)$ the corresponding pair and let $T$ be the maximal time of existence. Recall also that, by possibly subtracting a suitable function of time from $p_\varepsilon$, we can assume that $p_\varepsilon$ is the potential theoretic solution of 
\[
- \Delta p_\varepsilon = {\rm div}\, {\rm div}\, (u_\varepsilon \otimes u_\varepsilon)\, . 
\]
We first show that $T=\infty$. First of all note that, by a simple computation,
\[
\frac{d}{dt} \int |u_\varepsilon|^2 (x,t)\, dx = - 2 \int |(-\Delta)^{\sfrac{\alpha}{2}} u_\varepsilon|^2 (x,t)\, dx\, .
\]
In particular, if $T$ were finite, the $L^2$-norm of $u_\varepsilon$ would remain bounded up to time $T$. In turn, this ensures that
all $C^k$-norms of $u_\varepsilon \ast \phi_\varepsilon$ stay bounded. From linear theory and bootstrap arguments it follows then easily
that all $C^k$-norms of $u$ stay bounded as well and it is then easy to see that the solution can be continued after the time $T$.

Take a sequence $\varepsilon_k\downarrow 0$ and extract a subsequence, not relabeled, which converges weakly to $u\in L^2(\mathbb R^3\times [0,T])$.
Observe that $u_\varepsilon$ enjoy uniform estimates in $L^\infty ([0, \infty), L^2(\mathbb R^3))$ and $(-\Delta)^{\sfrac{\alpha}{2}} u_\varepsilon$ enjoy uniform estimates in $L^2 (\mathbb R^3 \times [0, \infty))$. In particular, for every fixed time $T$, $u_\varepsilon$ enjoy uniform estimates in
$L^2 ([0, T], H^1(\mathbb R^3))$ and thus, by interpolation
, $u_\varepsilon$ is bounded uniformly in $L^{\sfrac{10}{3}} (\mathbb R^3\times [0,T])$ for every $T< \infty$. In turn, using Calder\'on--Zygmund estimates for the pressure $p_\varepsilon$ we conclude uniform estimates in $L^{\sfrac{5}{3}} (\mathbb R^3\times [0,T])$. We can thus extract a subsequence so that $p_{\varepsilon_k}\rightharpoonup p$ locally in $L^{\sfrac{5}{3}} (\mathbb R^3\times [0,T])$. We next show that $u_{\varepsilon_k}$ converges strongly in $L^2_{\rm loc} (\mathbb R^3\times [0,T])$,
which in turn will imply that it converges locally strongly in $L^p (\mathbb R^3\times [0,T])$ for every $p< \frac{10}{3}$. Such strong convergence implies that
\begin{itemize}
\item $p_{\varepsilon_k}$ converge strongly in any $L^q_{\rm loc} (\mathbb R^3\times [0,T])$ for $q<\frac{5}{3}$ (Calder\'on--Zygmund estimates);
\item $(u, p)$ is a weak solution of the hyperdissipative Navier--Stokes.
\end{itemize}
Finally, the energy inequality follows from the lower semicontinuity of the dissipative right hand side 
\[
\int |(-\Delta)^{\sfrac{\alpha}{2}} u_{\varepsilon_k}|^2 (x,t)\, dx
\]
because for a.e. $t$ we can find a further subsequence so that $u_{\varepsilon_k} (\cdot, t)$ converges strongly in $L^2$ to $ u(\cdot, t)$. 

The strong convergence in $L^2 (\RR^3 \times [0,T])$ follows from a classical Aubin--Lions type argument, which we include for the reader's convenience. In order to simplify the notation we denote $u_{\varepsilon_k}$ by $w_k$ and without loss of generality we fix a time $T<\infty$. 
First of all, by Sobolev embeddings, $\|w_k\|_{L^2([0,T],L^6(\RR^3))}\le C$. 

Let $\varepsilon > 0$ be given. 
We want to show that $\exists\,N\in\NN$ such that $\|w_k-w_j\|_{L^2(\mathbb R^3 \times [0,T])}<\varepsilon$ for every $k,j\ge N$. Fix a standard mollifier $\varphi_\delta$ in the variable $x$ and observe that for any $t\in[0,\infty)$
\[
\|w_k(\cdot,t)-w_k\ast\varphi_\delta(\cdot,t)\|_{L^2 (\mathbb R^3)}\le C \delta^{\alpha} \|(-\Delta)^{\sfrac{\alpha}{2}}w_k(\cdot,t)\|_{L^2 (\mathbb R^3)}\,.
\]
In particular, for a fixed, sufficiently small $\delta$ we achieve 
\begin{equation}\label{e:1/3}
\|w_k\ast \varphi_\delta - w_k\|_{L^2 (\RR^3\times [0,T])} < \frac{\varepsilon}{3} \qquad \forall k\, .
\end{equation}
Next, mollifying the equation for $w_k$, for $\ell=1,2,3$, we find
\[
\partial_t (w_k*\varphi_\delta)_\ell = - \sum_{i=1}^3 f_{i,\ell,k} * \partial_{x_i} \varphi_\delta - w_k * (-\Delta)^{\alpha} \varphi_\delta\, ,
\]
where the functions $f_{i,\ell, k}$ are given by
\[
f_{i,\ell,k} := { (w_k)_i \ast \phi_{\varepsilon_k}} (w_k)_\ell + p \delta_{i\ell}\, 
\]
and thus enjoy uniform $L^{\sfrac{5}{3}}(\RR^3\times [0,T])$ bounds. 
Using the estimate $\|\zeta * \varphi_\delta \|_{W^{1, \infty} (\RR^3)} \leq C (\delta) \|\zeta\|_{L^1 (\mathbb R^3)}$ for each time slice, we conclude a bound
\[
\int_0^T \|\partial_t w_k * \varphi_\delta (\cdot, t)\|^{\sfrac{5}{3}}_{W^{1, \infty} (\mathbb R^3)}\, dt \leq C (\delta)\, ,
\]
where $C (\delta)$ is a constant depending upon $\delta$ but independent of $k$. 

So we can regard $[0,T]\ni t \mapsto w_k*\varphi_\delta (\cdot, t)$ as a sequence of equicontinuous and equibounded curves taking values in $W^{1, \infty} (\RR^3)$. Let $B_R$ be a (closed) ball of $W^{1,\infty} (\RR^3)$ so that the images of $w_k*\varphi_\delta$ are all contained inside it. If we endow $B_R$ with the $\|\cdot\|_\infty$-norm, then we have a compact metric space $X$. Hence we can regard $[0,T]\ni t \mapsto w_k*\varphi_\delta (\cdot, t)$ as an equicontinuous and equibounded sequence in the compact metric space $X$. By the Ascoli--Arzel\`a theorem the sequence is then precompact. Since the limit is unique (namely $u * \varphi_\delta$), we can conclude that the sequence $w_k*\varphi_\delta$ converges uniformly on $\RR^3 \times [0,T]$. 

Thus there exists $N$ large enough such that 
\[
\|w_k\ast\varphi_\delta-w_j\ast\varphi_\delta\|_{L^2(\RR^3 \times [0,T])}<\frac{\varepsilon}{3}\qquad \mbox{for all $k,j\geq N$.}
\]
Therefore, combining the latter inequality with \eqref{e:1/3}, for $j,k\geq N$ we have
\begin{align*}
\|w_k-w_j\|_{L^2 (\RR^3\times [0,T])} &\leq \|w_k - w_k\ast \varphi_\delta\|_{L^2 (\RR^3 \times [0,T])} + 
\|w_k\ast \varphi_\delta - w_j \ast \varphi_\delta\|_{L^2 (\RR^3\times [0,T])}\\
&\quad  + \|w_j- w_j\ast \varphi_\delta\|_{L^2 (\RR^3\times [0,T])}
< \varepsilon\, .
\end{align*}
This completes the proof of the strong convergence of $w_k$ and hence the proof of the existence part of Theorem \ref{t:Leray}. 

\subsection{The energy equality for smooth solutions}
We give a formal justification of the definition of suitable weak solutions by showing that \eqref{eqn:suit-weak-tested} holds with equality for smooth solutions of the hyperdissipative Navier-Stokes equations. We multiply the NS equation by $u\varphi$ and integrate to get
\begin{multline}\label{e:fs_energy}
	\int_{\RR^3} \varphi(x,t) |u(x,t)|^2 \, dx -
	\int_0^t \int_{\RR^3} \Big[ |u|^2 \partial_t \varphi + (| u |^2 +2p) u \cdot \nabla \varphi \Big] \, dx \, ds=-\int_0^t  \int_{\RR^3}\varphi u \cdot (-\Delta)^\alpha u \,dx \, ds\,.
\end{multline}
We manipulate the right hand side with $t$ fixed. Recalling \eqref{calc:lapl_b}, we notice that
\begin{align}
\overline{\rm div}\left(\overline\nabla(\varphi u_i^{*})y^b\overline\Delta_b u_i^*\right) & 
= \overline\Delta_b(\varphi u^*_i)y^b\overline\Delta_bu_i^*+\overline\nabla(\varphi u^*_i)y^b\overline\nabla\overline\Delta_b u_i^*\,.\label{bdc}
\end{align}
Integrating the left hand side, using the divergence theorem and recalling that $\partial_y\varphi=0$ when $y=0$, we obtain
\begin{equation*}
\int_{\RR^{4}_+}\overline{\rm div}\left(\overline\nabla(\varphi u_i^*)y^b\overline\Delta_b u_i^*\right) = - \int_{\RR^3} \partial_y (\varphi u_i^*)y^b \overline{\Delta}_b u_i^* = - \int_{\RR^3} \partial_y\varphi u_i^* y^b \overline{\Delta}_b u_i^*
- \int_{\RR^3} \varphi \partial_y u_i^* y^b \overline{\Delta}_b u_i^*
=0\,.\end{equation*}
 We integrate by parts the right hand side of \eqref{e:fs_energy}, by means of the divergence theorem
\begin{align}
\int_{\RR^3}\varphi u \cdot (-\Delta)^\alpha u \,dx & \stackrel{\eqref{eqn:frac-lap-est}}{=}c_{\alpha}\lim_{y\to 0}\int_{\RR^3}\varphi u_i^*(\cdot,0)y^b\partial_y\overline\Delta_b u_i^*(\cdot, 0)\,dx
=- c_{\alpha}\!\int_{\RR^4_+}\overline{\rm div}\left(\varphi u_i^*y^b\overline\nabla\overline\Delta_b u_i^*\right)\nonumber\\
& =-c_{\alpha}\! \int_{\RR^4_+}\varphi u_i^*y^b\overline\Delta_b^2 u_i^*-c_{\alpha}\!\int_{\RR^4_+} y^b \overline\nabla(\varphi u_i^*) \cdot \overline\nabla\overline\Delta_bu_i^*
 \stackrel{\eqref{biharm}\&\eqref{bdc}}{=}\! c_{\alpha}\!\int_{\RR^4_+}y^b\overline\Delta_b u_i^*\overline\Delta_b(u_i^*\varphi)\label{ll}
.
\end{align}
Observe that the commutator gives, for every $i=1,2,3$,
\begin{align*}
[\varphi,\overline\Delta_b]u_i^* :=\overline\Delta_b(\varphi u_i^* )-\varphi\overline\Delta_b u_i^*& = 2\overline\nabla\varphi\overline\nabla u_i^* +\overline\Delta_b\varphi u_i^*\,.
\end{align*}
Replacing the last line in \eqref{ll} with the commutator $[\varphi,\overline\Delta_b]$, we obtain
\begin{equation}\label{e:energia_liscia}
\begin{split}
\int_{\RR^3}\varphi u \cdot (-\Delta)^\alpha u \,dx
&= \int_{\RR^4_+}y^b |\overline\Delta_b u^*|^2 \varphi + \int_{\RR^4_+}[\varphi, \overline{\Delta}_b ]u_i^*y^b\overline\Delta u_i^* 
\\
&= c_{\alpha}\int_{\RR^4_+}\varphi y^b |\overline\Delta_b u^*|^2+ c_{\alpha}\int_{\RR^4_+}2y^b\overline\nabla\varphi\overline\nabla u_i^* \overline\Delta_b u_i^* + y^b u_i^* \overline\Delta_b\varphi\overline\Delta_b u_i^*\,.
\end{split}
\end{equation}

\subsection{Existence of suitable weak solutions}
In order to prove Theorem \ref{t:suitable} we proceed as in Section \ref{ss:Leray_ex}. Upon multiplying the mollified hyperdissipative Navier--Stokes equations, we derive an identity analogous to \eqref{e:energia_liscia} for the pair $(u_\varepsilon, p_\varepsilon)$. In particular, having fixed a test function $\varphi$, we have 
	\begin{align}
& {\int_{\RR^3} \varphi(\cdot,0,t) |u_\varepsilon (\cdot,t)|^2 \, dx}+c_{\alpha}{\int_0^t \int_{\RR^4_+} y^b|\overline \Delta_b u^*_\varepsilon |^2 \varphi \, dx \, dy \, ds}\nonumber\\
	=\; &\;
	\underbrace{\int_0^t \int_{\RR^3} \Big[ |u_\varepsilon|^2 \partial_t \varphi|_{y=0} + (| u_\varepsilon |^2 u_\varepsilon \ast \phi_\varepsilon+2p_\varepsilon u_\varepsilon) \cdot \nabla \varphi|_{y=0} \Big] 
	}_{=: C(\varepsilon)}
\;-c_{\alpha}\underbrace{\int_0^t\int_{\RR^4_+}y^b\left(2\overline\nabla\varphi\overline\nabla u_\varepsilon^* \overline\Delta_b u_\varepsilon^*+u_\varepsilon^* \overline\Delta_b\varphi\overline\Delta_b u_\varepsilon^* \right)
}_{=: D (\varepsilon)} \,.\label{eqn:suit-weak-tested_eps}
	\end{align}
Consider the subsequence $(u_{\varepsilon_k}, p_{\varepsilon_k})$ of the previous section, converging to the Leray--Hopf weak solution $(u,p)$.
By the strong convergence proved in the previous section, for a.e. $t\in(0,\infty)$ we conclude
\begin{align*}
& \lim_{k\to\infty} \int_{\RR^3} \varphi(\cdot,0,t) |u_{\varepsilon_k} (\cdot,t)|^2 \, dx = \int_{\RR^3} \varphi(\cdot,0,t) |u(\cdot,t)|^2 \, dx\\
&\lim_{k\to\infty} C (\varepsilon_k) =\int_0^t \int_{\RR^3} \Big[ |u|^2 \partial_t \varphi(\cdot,0,\cdot) + (| u|^2 u +2p u) \cdot \nabla \varphi(\cdot,0,\cdot) \Big] \, dx \, ds\,.
\end{align*}
Next observe that, by Theorem \ref{thm:yang} and the energy estimate of the previous section we know that
\[
\int_0^t \int_{\RR^4_+} y^b|\overline \Delta_b u^*_\varepsilon |^2 \, dx \, dy \, ds
\]
is uniformly bounded. Using the Poisson-type formula of Proposition \ref{prop:poisson_k} we see that $u^*_{\varepsilon_k}\to u^*$ strongly in $L^2 (\mathbb R^4_+ \times [0,t],y^b)$. Moreover, $\overline{\Delta}_b u^*_{\varepsilon_k}$ converges weakly in $L^2(\RR^4_+\times[0,t],y^b)$ to $\overline{\Delta}_b u^*$. Therefore, since $\varphi\geq 0$, by lower semicontinuity we have that
\[
\liminf_{k\to\infty} \int_0^t \int_{\RR^4_+} y^b|\overline \Delta_b u^*_{\varepsilon_k} |^2 \varphi \, dx \, dy \, ds \geq \int_0^t \int_{\RR^4_+} y^b|\overline \Delta_b u^*|^2 \varphi \, dx \, dy \, ds\, .
\]
Using the inequality \eqref{e:interp_2} for $u^*_{\varepsilon_k} - u^*= (u_{\varepsilon_k} - u)^*$ we infer
\[
\lim_{k\to \infty} \int_0^t \int_0^R \int_{B_R} y^b |\overline{\nabla} u^*_{\varepsilon_k} - \overline{\nabla} u^*|^2\, dx\, dy\, ds = 0\, .
\]
With these estimates we conclude
\[
\lim_{k\to\infty} D (\varepsilon_k) = \int_0^t\int_{\RR^4_+}y^b\left(2\overline\nabla\varphi\overline\nabla u^* \overline\Delta_b u^*+u^* \overline\Delta_b\varphi\overline\Delta_b u^* \right)\, dx \, dy \, ds.
\]

\subsection{Weak-strong uniqueness} In this section we complete the proof of Theorem \ref{t:Leray}. The proof is very similar to the same statement for the classical Navier--Stokes equations and we give it for the reader's convenience. Fix $u$ and $v$ Leray--Hopf weak solutions as in the statement. Observe that, by simple interpolation, $v\in L^4 (\mathbb R^3\times [0,T])$. Hence, multiplying the equation by $v$ and integrating in space (note that $(v\cdot \nabla) v\in L^{\sfrac{4}{3}}$ and thus $\nabla p \in L^{\sfrac{4}{3}}$ from Calder\'on-Zygmund estimates) the global energy inequality is indeed an identity. Simple computations show then that (in the distributional sense) for a.e. $t\in (0,\infty)$
\begin{align*}
\frac{d}{dt}\int_{\RR^3} \frac{|u-v|^2}{2}\, dx + \int_{\RR^3} |(-\Delta)^{\sfrac{\alpha}{2}} (u-v)|^2\, dx & \leq
\int_{\RR^3} |v| |u-v| |\nabla (u-v)|\, dx\\
& \le\frac{1}{2\varepsilon} \int_{\RR^3} |v|^2 |u-v|^2\, dx + \frac{\varepsilon}{2} \int_{\RR^3} |\nabla (u-v)|^2\, dx\\
& \le\frac{1}{2\varepsilon} \int_{\RR^3} |v|^2 |u-v|^2\, dx + C \varepsilon  \int_{\RR^3} |(-\Delta)^{\sfrac{\alpha}{2}} (u-v)|^2\, dx\,,
\end{align*}
where the last inequalities follow from Young inequality and Sobolev embeddings, respectively.
Fixing $\varepsilon$ small enough to reabsorb the second summand on the left hand side we conclude
\[
\frac{d}{dt} \int_{\RR^3} \frac{|u-v|^2}{2}\, dx \leq C \int_{\RR^3} |v|^2 |u-v|^2\, dx \leq \|v(\cdot, t)\|_{L^\infty} \int_{\RR^3} |u-v|^2\, dx\, .
\]
By Gronwall inequality, since $u$ and $v$ agree at the initial time, they are identically equal in $[0,T]$.

\section{The extension problem and a Poisson formula}\label{a:ext}

\begin{propos}\label{prop:poisson_k}
Let $\overline\Delta_b$ be the differential operator in \eqref{calc:lapl_b} with $\alpha\in (1,2)$ and $b=3-2\alpha$. The Poisson-type kernel
\begin{align}
P(x,y):=\frac{y^{2\alpha}}{\left(|x|^2+y^2\right)^{\frac{n+2\alpha}{2}}}\,,\label{e:Poisson}
\end{align}
satisfies then the following properties:
\begin{itemize}
\item $\overline{\Delta}_b \overline{\Delta}_b P =0$;
\item If $C_{n,\alpha}:=\|P(\cdot,1)\|_{L^1(\RR^n)}$, then
\begin{align}
\lim_{y\to 0} P(x,y) &=C_{n,\alpha}\delta_0\quad\qquad\text{weakly$^*$ as measures}\label{e:Dirichlet_P}\\
\lim_{y\to 0} y^{1-\alpha}\partial_y P(\cdot,y) &=0\quad\qquad\qquad\;\text{as distributions}\, .\label{e:Neumann_P}
\end{align} 
\end{itemize}
\end{propos}
\begin{remark}\label{rmk:propP} Let us notice the following properties of $P(x,y)$:
\begin{itemize}
\item[(i)] $\int_{\RR^n}P(x,y)\,dx$ is constant and therefore $\int_{\RR^n}\partial_y P(x,y)\,dx=0$;
\item[(ii)] $P(x,y)$ is rotation invariant and therefore $\int_{\RR^n}P(x,y) x\,dx=0$;
\item[(iii)] we have $\int_{\RR^n} P (x,y) x_i x_j \, dx =
y^2 \int_{\RR^n} \frac{z_i z_j}{(|z|^2 +1)^{\frac{n+2\alpha}{2}}}\, dz = \tilde c_{n,\alpha} y^2 \delta_{ij}$;
\item[(iv)] by the radial symmetry of $x\mapsto P (x,y)$, we have that $\int_{B_1} P (x,y) x_i x_j x_\ell \, dx = 0$.
\end{itemize}
\end{remark}
\begin{proof}
First of all, we use the notation $X=(x,y)\in\RR^{n+1}$ and observe that
\begin{equation*}
\overline\Delta |X|^{-\gamma} =(\gamma^2-(n-1)\gamma)|X|^{-(\gamma+2)},\qquad \partial_y|X|^{-\gamma} =-\gamma y|X|^{-(\gamma+2)}\,.
\end{equation*}
We can compute
\begin{align}
\overline\Delta_b\left(\frac{y^\beta}{|X|^\gamma}\right) = & \overline\Delta\left(\frac{y^\beta}{|X|^\gamma}\right)+\frac by\partial_y \left(\frac{y^\beta}{|X|^\gamma}\right)\nonumber\\
= & \beta(\beta-1)\frac{y^{\beta-2}}{|X|^{\gamma}}-2\gamma\beta\frac{y^{\beta}}{|X|^{\gamma+2}}+(\gamma^2-(n-1)\gamma)\frac{y^{\beta}}{|X|^{\gamma+2}}+b\beta\frac{y^{\beta-2}}{|X|^{\gamma}} -b\gamma\frac{y^{\beta}}{|X|^{\gamma+2}}\nonumber\\
= & \beta(\beta-1+b)\frac{y^{\beta-2}}{|X|^{\gamma}} +\gamma(\gamma-2\beta-(n-1)-b)\frac{y^{\beta}}{|X|^{\gamma+2}}\,.\label{e:db_gen}
\end{align}
When $\gamma=n+2\alpha$ and $\beta=2\alpha$ we see that
\begin{equation*}
\overline\Delta_b\left(\frac{y^{2\alpha}}{|X|^{n+2\alpha}}\right)=4\alpha\frac{y^{2\alpha-2}}{|X|^{n+2\alpha}} -2(n+2\alpha)\frac{y^{2\alpha}}{|X|^{n+2\alpha+2}}\,.
\end{equation*}
Taking $\overline\Delta_b$ and applying again \eqref{e:db_gen} to each summand of the right hand side, we obtain
\[
\overline\Delta_b\overline\Delta_b\left(\frac{y^{2\alpha}}{|X|^{n+2\alpha}}\right)=8\alpha(n+2\alpha)\frac{y^{2\alpha-2}}{|X|^{n+2\alpha+2}}-2(n+2\alpha)4\alpha\frac{y^{2\alpha-2}}{|X|^{n+2\alpha+2}}=0\,.
\]
Let us notice that, by a change of variable, $\|P(\cdot,y)\|_{L^1(\RR^n)}$ is independent from $y$. Moreover, for every $\varepsilon>0$, $\|P(\cdot,y)\|_{L^1(\RR^n\setminus B_\varepsilon)}\to 0$ as $y\to 0$, which therefore shows \eqref{e:Dirichlet_P}. 
{We next claim that, for any test function $\phi\in C^2 (\RR^n)$, 
\begin{equation}\label{e:stima_C-2}
\left|\int_{\RR^n} \partial_y P (x,y)\,\phi (x)\, dx\right| \leq C y \|\phi\|_{C^2(\RR^n)}\qquad\forall\, y>0\,.
\end{equation}
Multiplying by $y^{1-\alpha}$ the latter inequality and using $2-\alpha >0$, we easily that
\[
\lim_{y\downarrow 0} y^{1-\alpha} \int_{\RR^n} \partial_y P (x,y)\,\phi (x)\, dx = 0\, ,
\]
which, by the arbitrariness of the test function, implies  \eqref{e:Neumann_P}.

In order to show \eqref{e:stima_C-2}, we use the properties (i) and (ii) of Remark \ref{rmk:propP} to compute
\begin{align*}
J(y)&:=\int_{\RR^n}\partial_yP(x,y)\phi(x)\,dx
=\int_{\RR^n}\partial_y P(x,y)(\phi(x)-\phi(0)-\nabla\phi(0)\cdot x)\,dx\,,
\end{align*}
By means of a Taylor expansion, we obtain
\begin{align}
|J(y)| & \le C\|\phi\|_{C^2 (\RR^n)} \int_{\RR^n} \left|{\partial_y} P (x,y)\right||x|^2\,dx\label{e:prima_quadratica}
\end{align}
In order to bound the latter integral, but also for later purposes, we remark the following simple bounds:
\begin{align}
|\partial_y^k P (x,y)| &\leq C (k) \sum_{i=0}^k \frac{y^{2\alpha - k+2i}}{(|x|^2 +y^2)^{\frac{n+2\alpha}{2}+i}}
\leq C(k) \frac{y^{2\alpha - k}}{(|x|^2+y^2)^{\frac{n+2\alpha}{2}}}
\qquad \mbox{for $k\in \mathbb N$.}\label{e:puntuale}
\end{align}
Combining the latter estimate with \eqref{e:prima_quadratica} we obtain
\begin{align*}
|J (y)| & \le C\|\phi\|_{C^2 (\RR^n)} \int_{\RR^n} \frac{|x|^2y^{2\alpha -1}}{(|x|^2+y^2)^{\frac{n+2\alpha}{2}}}\,dx\ =
C\|\phi\|_{C^2 (\RR^n)} y \int_{\RR^n} \frac{|z|^2}{(1+|z|^2)^{\frac{n+2\alpha}{2}}}\, dz\, ,
\end{align*}
which implies \eqref{e:stima_C-2}.}
\end{proof}

For the sake of completeness, we add here a proof of Yang's theorem \ref{thm:yang} by means of the Poisson kernel computed above.
\begin{proof}[Proof of Theorem \ref{thm:yang}]
We define the extension $u^*\in L^2_{\rm loc}(\RR^{n+1}_+,y^b)$ as
\begin{equation}\label{e:rappresentazione}
u^*(x,y):= P(\cdot,y)\ast u (x)=\int_{\RR^n} \frac{y^{2\alpha}u(\xi)}{\left(|x-\xi|^2+y^2\right)^\frac{n+2\alpha}{2}}\,d\xi\,.
\end{equation}
Observe that $\overline\Delta_b u^*\in L^2(\RR^{n+1}_+,y^b)$. \eqref{biharm} is an obvious consequence of $\overline\Delta_b^2 P =0$. The boundary conditions \eqref{e:Dirichlet} and \eqref{e:Neumann} follow instead from \eqref{e:Dirichlet_P} and \eqref{e:Neumann_P}, respectively. Since the proof is entirely analogous, we just show the one for \eqref{e:Neumann}. {Fix a smooth test function
 $\varphi \in C^\infty_c(\RR^n)$ and observe that
\[
\int_{\RR^n\times\{y\}}y^{1-\alpha} \partial_y u^* f = \int_{\RR^n \times \{y\}} \left(y^{1-\alpha} \partial_y P (\cdot, y)\right)\ast u f 
= \int_{\RR^n\times\{y\}} u  \left(y^{1-\alpha} \partial_y P (\cdot, y)\right)\ast f\,.
\]
In particular, from \eqref{e:Neumann_P} and the smoothness of $f$, it follows that, as $y\downarrow 0$, $\left(y^{1-\alpha} \partial_y P (\cdot, y)\right)\ast f$ converges to $0$ in the Schwartz space $\mathscr{S}$. Since $u\in L^2$, we easily conclude that
\[
\lim_{y\downarrow 0} \int_{\RR^n\times\{y\}}y^{1-\alpha}\partial_y u^* f  = 0\, .
\]}

Since the functional $v\mapsto\int_{\RR^{n+1}_+}y^b|\overline\Delta_b u^*|^2$ is strictly convex, from the fulfillment of the {Euler-Lagrange conditions \eqref{biharm} and \eqref{e:Neumann}}, we obtain that $u^*$ is the unique minimizer for the functional, hence \eqref{e:minimum_yang_ext}. 

We next come to \eqref{e:en_of_ext}. We first observe that $\widehat{u^*}(\xi,y)=\hat{P}(\xi,y)\hat u(\xi)$. Since $P$ (hence $\hat P$) is radially symmetric, with a change of variables we have that $\hat P(\xi,y)=\widehat{P(\cdot,1)}(\xi y)=\widehat{P(\cdot,1)}(|\xi|y)$. For the sake of brevity, we call $\phi(z):=\widehat{P(\cdot,1)}(z)$. By Parseval's identity we can compute
\begin{align}
\int_{\RR^{n+1}_+}y^b|\overline\Delta_b u^*|^2 & =\int_{\RR^{n+1}_+} y^b\left||\xi|^2\widehat{u^*}+\partial_{yy}\widehat{u^*}+\frac by\partial_y\widehat{u^*}\right|^2\,d\xi\,dy\nonumber\\
 & = \int_{\RR^{n+1}_+} y^b|\xi|^4\left|\phi(|\xi|y)\hat u(\xi)+\frac {b}{|\xi|y}\phi'(|\xi|y)\hat u(\xi)+\phi''(|\xi|y)\hat u(\xi)\right|^2\,d\xi\,dy\nonumber\\
 & = \int_{\RR^{n+1}_+} z^b|\xi|^{2\alpha}\left|\phi(z)\hat u(\xi)+\frac {b}{z}\phi'(z)\hat u(\xi)+\phi''(z)\hat u(\xi)\right|^2\,d\xi\,dz\nonumber\\
 & = \int_{\RR^n}|\xi|^{2\alpha}|\hat u(\xi)|^2\, d\xi \int_{\RR_+} z^b\left|\phi(z)+\frac {b}{z}\phi'(z)+\phi''(z)\right|^2\,dz= c_{n,\alpha}^{-1} \int_{\RR^n}|\xi|^{2\alpha}|\hat u(\xi)|^2\, d\xi \,.\label{id:functionals}
\end{align}
This easily shows \eqref{e:en_of_ext}, but observe indeed that, by the very same argument, we can conclude
\begin{equation}\label{e:polarizzata}
\int_{\RR^{n+1}_+} y^b \overline\Delta_b v^* \overline\Delta_b u^* = c_{n,\alpha}^{-1} \int_{\RR^n} v (-\Delta)^\alpha u
\end{equation}
for every pair $u, v \in H^\alpha$.

\medskip

In order to complete the proof of the theorem, we wish to show \eqref{eqn:frac-lap-est}. First of all we claim that the family  of distributions $\{y^b \partial_y \overline{\Delta}_b P (\cdot, y)\}_{y\in (0,1)}$ is equibounded in the space $\mathscr{S}'$. In particular we will show the existence of a constant $C$ such that
\begin{equation}\label{e:bound_C-4}
\left| \int_{\RR^n} y^b \phi (x) \partial_y \overline{\Delta}_b P (x,y)\, dx \right| \leq C \|\phi\|_{C^4(\RR^n)}
\qquad \forall y\in (0,1) \;\mbox{and}\; \forall \phi \in \mathscr{S}\, .
\end{equation}
To prove the latter bound we first compute
\[
y^b \partial_y \overline{\Delta}_b P = y^b \Delta \partial_y P + y^b \partial_y^3 P + y^b \partial_y \left(\frac{b}{y} \partial_y P\right) \, .
\]
We thus need to bound the three distributions $y^b \Delta \partial_y P (\cdot, y), y^b \partial^3_y P (\cdot, y)$ and $y^b \partial_y (y^{-1} \partial_y P)$. Recalling the estimate \eqref{e:stima_C-2} of the previous proposition, we easily see that 
\[
\left| \int_{\RR^n} y^b \partial_y \Delta P (x, y) \phi (x)\, dx\right|
= \left| \int_{\RR^n} y^b \partial_y P (x, y) \Delta \phi (x)\, dx\right| \leq C y^{4-2\alpha} \|\Delta \phi\|_{C^2(\RR^n)}\, .
\]
{Thus, we just need to bound the remaining two terms $y^b \partial^3_y P (\cdot, y)$ and $y^b \partial_y (y^{-1} \partial_y P)$.}

By Remark \ref{rmk:propP} we achieve that, for every $y>0$,
\[
\int_{\RR^n} \partial_y^3 P\, dx = \int_{\RR^n} \partial_y^3 P x_i\, dx = \int_{\RR^n} \partial_y^3 P x_i x_j\, dx 
= \int_{B_1} \partial^3_y P x_i x_j x_\ell\, dx = 0
\]
and
\begin{multline*}
\int_{\RR^n} \partial_y (y^{-1} P)\, dx = \int_{\RR^n} \partial_y (y^{-1} \partial_y P) x_i\, dx = \int_{\RR^n} \partial_y (y^{-1} \partial_y P) x_i x_j\, dx= \;\int_{B_1} \partial_y (y^{-1} \partial_y P) x_i x_j x_\ell\, dx = 0\, .
\end{multline*}
Thus we can compute
\begin{multline*}
 \int_{\RR^n} \partial_y^3 P (x,y) \phi (x)\, dx\\
= \int_{\RR^n} \partial_y^3 P (x,y) \left(\phi (x) - \phi (0) - \nabla \phi (0) \cdot x 
- \frac{1}{2} \sum_{i,j} \partial^2_{ij}\phi  (0) x_i x_j - \frac{1}{6} \sum_{i,j,\ell} \partial^3_{i,j,\ell} \phi (0) x_i x_j x_\ell \mathbbm{1}_{B_1} (x)\right)\, dx\, .
\end{multline*}
From the Taylor expansion for $\phi$ we conclude
\[
\left|\phi (x) - \phi (0) - \nabla \phi (0) \cdot x 
- \frac{1}{2} \sum_{i,j} \partial^2_{ij}\phi  (0) x_i x_j - \frac{1}{6} \sum_{i,j,\ell} \partial^3_{i,j,\ell} \phi (0) x_i x_j x_\ell \mathbbm{1}_{B_1} (x)\right| \leq C\|\phi\|_{C^4(\RR^n)} \min \{|x|^4, |x|^2\}\, .
\]
{From \eqref{e:puntuale} we then get 
\begin{align*}
\left| \int_{\RR^n} y^b \partial_y^3 P (x,y) \phi (x)\, dx\right| &\leq C \|\phi\|_{C^4(\RR^n)} \int_{B_1} \frac{|x|^4}{(|x|^2 +y^2)^{\frac{n+2\alpha}{2}}}\,dx + C\|\phi\|_{C^4(\RR^n)} \int_{\{|x|>1\}} \frac{|x|^2}{(|x|^2+y^2)^{\frac{n+2\alpha}{2}}}\\
&\leq C\|\phi\|_{C^4 (\RR^n)} \int_{B_1} |x|^{-n+ (4-2\alpha)}\, dx + C \|\phi\|_{C^4 (\RR^n)} \int_{\{|x|>1\}} |x|^{-n-(2-2\alpha)}\, dx\\
&\leq C \|\phi\|_{C^4 (\RR^n)}\, .
\end{align*}}
\noindent The same bound with $y^b \partial_y (y^{-1} \partial_y P (\cdot, y))$ replacing $y^b \partial_y^3 P (\cdot, y)$ is entirely analogous and we have thus shown \eqref{e:bound_C-4}. 

\medskip

Fix now $\varphi$ and $u$ both in $\mathscr{S}$. We will show below that
\begin{equation}\label{e:convergenza}
\lim_{y\to 0} \int_{\RR^n\times\{y\}} \varphi^* (x, y) y^b \partial_y \overline{\Delta}_b  u^* (x,y)\, dx =
c \int_{\RR^n} \varphi (x) (-\Delta)^\alpha u (x)\, dx \, ,
\end{equation}
for some geometric constant $c= c (n, \alpha)$.

Since $\varphi$ is smooth, $\varphi^* (\cdot, y) \to \varphi$ in $\mathscr{S}$ and, in particular, thanks to the bound 
\eqref{e:bound_C-4}, \eqref{e:convergenza} implies
\[
\lim_{y\to 0} \int_{\RR^n\times\{y\}} \varphi (x) y^b \partial_y \overline{\Delta}_b u^* (x,y)\, dx =
c \int_{\RR^n} \varphi (x) (-\Delta)^\alpha u (x)\, dx\, .
\]
Since $\varphi\in \mathscr{S}$ is arbitrary, the latter implies that $y^b \partial_y \overline{\Delta}_b u^* (\cdot,y)
\to c (-\Delta)^\alpha u$ in the sense of distributions. In turn, since $u\in \mathscr{S}$ is also arbitrary and 
$y^b \partial_y \overline{\Delta}_b u^* (\cdot,y) = (y^b \partial_y \overline{\Delta}_b P (\cdot,y)) \ast u$,
we conclude that
\[
\lim_{y\downarrow 0} y^b \partial_y \overline{\Delta}_b P (\cdot,y) = c (-\Delta)^\alpha \delta_0
\]
in the sense of distributions. Finally, the latter identity implies \eqref{eqn:frac-lap-est} for a general tempered distribution $u$
(and thus also for $u\in H^\alpha$).

We are thus left to show \eqref{e:convergenza}. Observe first that
\begin{align}
\int_{\RR^n\times\{y=h\}} y^b \varphi^*\partial_y\overline\Delta_bu^* & =\int_{\RR^{n+1}\cap\{y> h\}}\overline{\rm div}\left(y^b \varphi^*_i\overline \nabla \overline\Delta_bu^*_i\right) \stackrel{\eqref{biharm}}{=}\int_{\RR^{n+1}\cap\{y> h\}} y^b\overline\nabla \varphi^*_i \overline\nabla\overline\Delta_b u^*_i\nonumber\\
& =\underbrace{\int_{\RR^n\times\{y=h\}}y^b\partial_y \varphi^*\cdot \overline\Delta_b u^*}_{=: I(h)}+\int_{\RR^{n+1}\cap\{y> h\}} \overline{\rm div}(y^b\overline\nabla \varphi^*)\overline\Delta_b u^*\,.\label{tough_after}
\end{align}
The key point is that $I(h)\to 0$ as $h\downarrow 0$. The latter claim implies then
\[
\lim_{y\to 0} \int_{\RR^n} \varphi^* (x, y) y^b \partial_y \overline{\Delta}_b \overline u^* (x,y)\, dx
= \int_{\RR^{n+1}_+} y^b \overline{\Delta}_b \varphi^* \overline{\Delta}_b u^*\, ,
\]
which by \eqref{e:polarizzata} gives \eqref{e:convergenza}. 

In order to show that $I(h)\to 0$, observe that we have already argued that $y^{1-\alpha}\partial_y \varphi^*$ converges to zero in the Schwartz space $\mathscr{S}$ (because $\varphi\in \mathscr{S}$). We thus need to show that $\{y^{2-\alpha}  \overline{\Delta}_b u^* (\cdot, y)\}_{y\in (0,1)}$ is bounded in the space $\mathscr{S}'$. Using again the representation through the Poisson's kernel, it suffices to show the boundedness of $\{y^{2-\alpha} \overline{\Delta}_b P (\cdot, y)\}_{y\in (0,1)}$. We compute
\[
y^{2-\alpha} \overline{\Delta}_b P (\cdot, y) = y^{2-\alpha} \Delta P (\cdot, y) + y^{2-\alpha} \partial_y^2 P (\cdot, y) + b y^{1-\alpha} \partial_y P (\cdot, y)\, .
\]
Since $2-\alpha>0$ and $P (\cdot, y) \to c \delta_0$, the boundedness of the first summand is obvious. In the previous proposition we have already shown that $y^{1-\alpha} \partial_y P (\cdot, y) \to 0$. We thus need to handle the second summand. We proceed as in the above arguments and use the momentum conditions (see Remark \ref{rmk:propP})
\[
\int_{\RR^n} \partial_y^2 P (x,y)\, dx = \int_{\RR^n} \partial_y^2 P (x,y) x_i\, dx = 0\, 
\]
to write
\[
\int_{\RR^n}\partial_y^2 P (x,y) \phi (x)\, dx = \int_{\RR^n} \partial_y^2 P (x,y) (\phi (x) - \phi (0) - \nabla \phi (0)\cdot x)\, dx\, .
\]
{In particular, by the Taylor expansion and the estimates \eqref{e:puntuale}, 
\begin{align}
\left|\int_{\RR^n} y^{2-\alpha} \partial_y^2 P (x,y)\, \phi (x)\, dx\right| &\leq C \|\phi\|_{C^2(\RR^n)} y^{2-\alpha} \int_{\RR^n} \left|\partial_y^2 P 
\right||x|^2\, dx \leq C\|\phi\|_{C^2(\RR^n)} y^{\alpha} \int_{\RR^n} \frac{|x|^2}{(|x|^2 + y^2)^{\frac{n+2\alpha}{2}}}\, dx\\
& \leq C\|\phi\|_{C^2(\RR^n)} y^{2-\alpha} \int_{\RR^n} \frac{|z|^2}{(|z|^2 + 1)^{\frac{n+2\alpha}{2}}}\, dz\le C \|\phi\|_{C^2(\RR^n)}\, .\qedhere
\end{align}
}
\end{proof}

\bibliographystyle{plain}
\bibliography{NS-alfa_C}
\end{document}